\documentclass[10pt]{amsart}
\usepackage{amscd}
\usepackage{amssymb}
\usepackage[all]{xy}
\usepackage{lmodern} 
\usepackage[T1]{fontenc}
\date{31 August 2012}
\title{On the Homology of Completion and Torsion}

\author{Marco Porta, Liran Shaul and Amnon Yekutieli}

\address{Department of  Mathematics,
Ben Gurion University, Be'er Sheva 84105, Israel}

\address{ {\em E-mail address}:  (Porta) {\tt marcoporta1@libero.it},
(Shaul) {\tt shlir@math.bgu.ac.il},
(Yekutieli) {\tt amyekut@math.bgu.ac.il} } 

\thanks{
Primary: 13D07; Secondary: 13B35, 13C12, 13D09, 18E30.}
\thanks{This research was supported by the Israel Science Foundation and the 
Center for Advanced Studies at BGU}

\newtheorem{thm}[equation]{Theorem}
\newtheorem{cor}[equation]{Corollary}
\newtheorem{prop}[equation]{Proposition}
\newtheorem{lem}[equation]{Lemma}
\theoremstyle{definition}
\newtheorem{dfn}[equation]{Definition}
\newtheorem{rem}[equation]{Remark}
\newtheorem{exa}[equation]{Example}

\numberwithin{equation}{section}

\newcommand{\iso}{\xrightarrow{\simeq}}
\newcommand{\inj}{\hookrightarrow}

\newcommand{\xar}{\xrightarrow}
\newcommand{\opn}{\operatorname}
\newcommand{\cat}[1]{\operatorname{\mathsf{#1}}}

\newcommand{\rmitem}[1]{\item[\text{\textup{(#1)}}]}
\newcommand{\mfrak}[1]{\mathfrak{#1}}
\newcommand{\mcal}[1]{\mathcal{#1}}

\newcommand{\mrm}[1]{\mathrm{#1}}
\newcommand{\mbb}[1]{\mathbb{#1}}

\newcommand{\tup}[1]{\textup{#1}}
\newcommand{\bsym}[1]{\boldsymbol{#1}}

\newcommand{\til}[1]{\tilde{#1}}
\newcommand{\what}[1]{\widehat{#1}}

\newcommand{\K}{\mbb{K}}

\newcommand{\N}{\mbb{N}}
\newcommand{\Z}{\mbb{Z}}

\newcommand{\La}{\Lambda}
\newcommand{\Ga}{\Gamma}
\renewcommand{\a}{\mfrak{a}}
\renewcommand{\b}{\mfrak{b}}
\renewcommand{\c}{\mfrak{c}}
\newcommand{\p}{\mfrak{p}}

\newcommand{\al}{\alpha}
\newcommand{\be}{\beta}

\renewcommand{\d}{\mrm{d}}
\newcommand{\ot}{\otimes}

\newcommand{\lb}{\linebreak}
\newcommand{\distri}{\xar{ \vartriangle }}
\renewcommand{\ot}{\otimes}

\begin{document}

\begin{abstract}
Let $A$ be a commutative ring, and $\a$ a {\em weakly proregular}
ideal in $A$. This includes the noetherian case: if $A$ is noetherian then any
ideal in it is weakly proregular; but there are other interesting examples. In
this paper we prove the {\em MGM equivalence}, which is an
equivalence between the category of {\em cohomologically $\a$-adically complete
complexes} and the category of {\em cohomologically $\a$-torsion complexes}.
These are triangulated subcategories of the derived category of $A$-modules. Our
work extends earlier work by Alonso-Jeremias-Lipman, Schenzel and
Dwyer-Greenlees.
\end{abstract}

\maketitle
\tableofcontents

\setcounter{section}{-1}
\section{Introduction}

Let $A$ be a commutative ring, and let $\a$ be an ideal in it. 
(We do not assume that $A$ is noetherian or $\a$-adically complete.)
There are two operations associated to this data: the {\em $\a$-adic
completion} and the {\em $\a$-torsion}. 
For an $A$-module $M$ its $\a$-adic completion is the $A$-module
\[ \Lambda_{\a} (M) = \what{M} := \lim_{\leftarrow i} \,
M / \a^i M . \]
An element $m \in M$ is called an $\a$-torsion element if
 $\a^i m = 0$ for $i \gg 0$. The $\a$-torsion elements form 
the $\a$-torsion submodule $\Gamma_{\a} (M)$ of $M$.

Let us denote by $\cat{Mod} A$ the category of $A$-modules. So we have additive
functors 
\[ \Lambda_{\a}, \Gamma_{\a} : \cat{Mod} A \to \cat{Mod} A . \]
The functor $\Gamma_{\a}$ is left exact; whereas $\Lambda_{\a}$ is neither left
exact nor right exact. (Of course when $A$ is noetherian, the completion functor
$\Lambda_{\a}$ is exact on the subcategory $\cat{Mod}_{\mrm{f}} A$ of finitely
generated modules.) In this paper we study several questions of homological
nature about these two functors. 

The derived category of $\cat{Mod} A$ is denoted by $\cat{D}(\cat{Mod} A)$. 
As explained in Section \ref{sec:prelim}, the derived functors
\[ \mrm{L} \Lambda_{\a}, \mrm{R} \Gamma_{\a} : 
\cat{D}(\cat{Mod} A) \to \cat{D}(\cat{Mod} A) \]
exist. The left derived functor $\mrm{L} \Lambda_{\a}$ is constructed using
K-projective resolutions, and the right derived functor $\mrm{R} \Gamma_{\a}$
is constructed using K-injective resolutions.

The functor $\mrm{R} \Gamma_{\a}$ has been studied in great length already in
the 1950's, by Grothendieck and others (in the context of local cohomology). 

The left derived functors $\mrm{L}^i \Lambda_{\a}$ were studied by 
Matlis \cite{Ma2} and Greenlees-May \cite{GM}. 
The first treatment of the total left derived functor $\mrm{L} \Lambda_{\a}$ was
in the paper \cite{AJL1} by Alonso-Jeremias-Lipman from 1997.
In this paper the authors established the {\em Greenlees-May Duality}, which we
find deep and remarkable. 
The setting in \cite{AJL1} is geometric: the completion of a non-noetherian
scheme along a proregularly embedded closed subset. 
However, certain aspects of the theory remained unclear (see Remarks
\ref{rem:280} and \ref{rem:historical}). One of our aims in
this paper is to clarify the foundations of the theory in the algebraic
setting. We also extend the scope of the existing results. 

Two other, much more recent papers also influenced our work. In the paper
\cite{KS2} of Kashiwara-Schapira there is a part devoted to what they call
{\em cohomologically complete complexes}. We wondered what might be the relation
between this notion and the derived completion functor $\mrm{L} \Lambda_{\a}$. 
The answer we discovered is Theorem \ref{thm:50} below. 

The paper \cite{Ef} by Efimov describes an operation of {\em completion  by
derived double centralizer}. This idea is attributed to Kontsevich. A similar
results was obtained in \cite{DGI}. Our interpretation of this completion
operation is in the companion paper \cite{PSY1}, and it relies on the work in
this paper.

Let us turn to the results in our paper. 
We work in the following context: $A$ is a commutative ring, and $\a$ is a
{\em weakly proregular ideal} in it. By definition an ideal is weakly
proregular if it can be generated by a {\em weakly proregular sequence}
$\bsym{a} = (a_1, \ldots, a_n)$ of elements of $A$. The definition of
proregularity for sequences
is a bit technical (see Definition \ref{dfn:250}). 

It is important to know that if $\a$ is a weakly proregular ideal in $A$,
then any finite sequence that generates $\a$ is weakly proregular.
If $A$ is noetherian then every finite sequence in $A$ is
weakly proregular, so that every ideal in $A$ is weakly proregular.
These results were already proved in \cite{AJL1} and \cite{Sc}. 
We provide short proofs for the benefit of the reader (see
Corollary \ref{cor:290} and Theorem \ref{thm:253} in the body of the paper). 
We also give a fairly natural example of a weakly proregular sequence in a
non-noetherian ring (Example \ref{exa:251}).

A complex $M \in \cat{D}(\cat{Mod} A)$ is called a {\em cohomologically
$\a$-torsion complex} if the canonical morphism 
$\mrm{R} \Gamma_{\a} (M) \to M$
is an isomorphism.
The complex $M$  is called a {\em cohomologically $\a$-adically complete
complex} if the canonical morphism 
$M \to \mrm{L} \Lambda_{\a} (M)$
is an isomorphism.
We denote by $\cat{D}(\cat{Mod} A)_{\a \tup{-tor}}$ and 
$\cat{D}(\cat{Mod} A)_{\a \tup{-com}}$
the full subcategories of $\cat{D}(\cat{Mod} A)$ consisting of
cohomologically $\a$-torsion complexes and cohomologically $\a$-adically
complete complexes, respectively. These are triangulated subcategories.

Here is the main result of our paper. 

\begin{thm}[MGM Equivalence] \label{thm:49}
Let $A$ be a commutative ring, and $\a$ a weakly proregular ideal in it.
\begin{enumerate}
\item For any $M \in \cat{D}(\cat{Mod} A)$ one has
$\mrm{R} \Gamma_{\a} (M) \in \cat{D}(\cat{Mod} A)_{\a \tup{-tor}}$
and 
$\mrm{L} \Lambda_{\a} (M) \in \cat{D}(\cat{Mod} A)_{\a \tup{-com}}$.

\item The functor 
\[ \mrm{R} \Gamma_{\a} : 
\cat{D}(\cat{Mod} A)_{\a \tup{-com}} \to 
\cat{D}(\cat{Mod} A)_{\a \tup{-tor}} \]
is an equivalence, with quasi-inverse $\mrm{L} \Lambda_{\a}$. 
\end{enumerate}
\end{thm}

This is repeated as Theorem \ref{thm:26} in the body of the paper.
The letters ``MGM'' stand for Matlis, Greenlees and May.

Similar results can be found in \cite{AJL1, Sc, DG}, and possibly
some weaker version of Theorem \ref{thm:49} can be deduced from these
results. But as far as we can tell,
Theorem \ref{thm:49} is new. See Remarks  \ref{rem:280} and \ref{rem:historical}
for a discussion. The main ingredient in the proof of
the MGM equivalence is Theorem \ref{thm:270} below. 

Given a finite sequence $\bsym{a}$ that generates $\a$, we construct
explicitly a complex 
$\opn{Tel}(A; \bsym{a})$, called the {\em telescope complex}. It is a bounded
complex of countable rank free $A$-modules. There is a functorial homomorphism
of complexes (also with explicit formula) 
\[ \opn{tel}_{\bsym{a}, M} : 
\opn{Hom}_A \bigl( \opn{Tel}(A; \bsym{a}), M \bigr) \to \Lambda_{\a} (M) \]
for any $M \in \cat{Mod} A$. By totalization we get a homomorphism 
$\opn{tel}_{\bsym{a}, M}$ for any 
$M \in \cat{C}(\cat{Mod} A)$.
See Definitions \ref {dfn:5} and \ref{dfn:280}. 

\begin{thm} \label{thm:270}
Let $A$ be a commutative ring, let $\bsym{a}$ be a weakly proregular sequence
in $A$, and let $\a$ be the ideal generated by $\bsym{a}$.
If $P$ is a K-flat complex of $A$-modules, then the homomorphism
\[ \opn{tel}_{\bsym{a}, P} : 
\opn{Hom}_A \bigl( \opn{Tel}(A; \bsym{a}), P \bigr) \to \Lambda_{\a} (P) \]
is a quasi-isomorphism.
\end{thm}

This is Corollary \ref{cor:294} in the body of the paper. 
The concept of telescope complex is not new of course, but our treatment 
appears to be quite different from anything we saw in the literature. 

Along the way we also prove that the functors $\mrm{R} \Gamma_{\a}$ and 
$\mrm{L} \Lambda_{\a}$ have finite cohomological dimensions. 
(An upper bound is the minimal length of a sequence 
that generates the ideal $\a$.) This implies that 
\begin{equation} \label{eqn:50}
\cat{D}(\cat{Mod} A)_{\a \tup{-tor}} = 
\cat{D}_{\a \tup{-tor}}(\cat{Mod} A) ,
\end{equation}
the latter being the subcategory of $\cat{D}(\cat{Mod} A)$ consisting of
complexes with $\a$-torsion cohomology modules (see Corollary \ref{cor:17}).
Note that such a statement for $\cat{D}(\cat{Mod} A)_{\a \tup{-com}}$
is false: in Example \ref{exa:1} we exhibit a cohomologically $\a$-adically
complete complex $P$ such that $\mrm{H}^i (P) = 0$ for all $i \neq 0$, and the
module $\mrm{H}^0 (P)$ is {\em not} $\a$-adically complete. 

Let $\bsym{a} = (a_1, \ldots, a_n)$ be a generating sequence for the ideal 
$\a$. In Section \ref{sec:der-loc} we construct a
noncommutative DG $A$-algebra $\opn{C}(A; \bsym{a})$, which we call the {\em
derived localization of $A$ with respect to $\bsym{a}$}. 
When $n = 1$ (we refer to this as the principal case, since the ideal
$\a$ is principal) then 
$\opn{C}(A; \bsym{a}) = A[a_1^{-1}]$, the usual localization. For $n > 1$ the
construction uses the \v{C}ech cosimplicial algebra and the Alexander-Whitney
multiplication.

\begin{thm} \label{thm:50}
Let $A$ be a commutative ring, $\bsym{a}$ a weakly proregular sequence in $A$,
and $\a$ the ideal generated by $\bsym{a}$.
The following conditions are equivalent for 
$M \in \cat{D}(\cat{Mod} A)$\tup{:}
\begin{enumerate}
\rmitem{i} $M$ is  cohomologically $\a$-adically complete.
\rmitem{ii} $\opn{RHom}_A \bigl( \opn{C}(A; \bsym{a}) , M \bigr) = 0$.
\end{enumerate}
\end{thm}

This is Theorem \ref{thm:263} in the body of the paper. 
The principal regular case ($n= 1$ and $a_1$ a non-zero-divisor)
was considered by Kashiwara-Schapira in \cite{KS2}. Indeed, in \cite{KS2}
condition (ii) was used as the definition of cohomologically complete
complexes. After hearing about the results of \cite{KS2}, we wondered whether
they hold in greater generality (for $n > 1$ and no regularity assumption on
the sequence $\bsym{a}$). More in this direction can be found in the companion
paper \cite{PSY2}.

\medskip \noindent
\textbf{Acknowledgments.}
We wish to thank Bernhard Keller, John Greenlees, Joseph Lipman, Ana Jeremias,
Leo Alonso and Peter Schenzel for helpful discussions.
Thanks also to the referee for his/her comments on improving the paper.

\section{Preliminaries on Homological Algebra}
\label{sec:prelim}

This paper relies on delicate work with derived functors. Therefore we begin
with a review of some facts on homological algebra. There are also a few
new results. By default all rings considered in the paper are commutative. 

Let $\cat{M}$ be an abelian category. 
As in \cite{RD} we denote by $\cat{C}(\cat{M})$ the category of complexes of
objects of $\cat{M}$, by $\cat{K}(\cat{M})$ its homotopy category, 
and by $\cat{D}(\cat{M})$ the derived category.
There are full subcategories  $\cat{D}^-(\cat{M})$, $\cat{D}^+(\cat{M})$ and 
$\cat{D}^{\mrm{b}}(\cat{M})$ of $\cat{D}(\cat{M})$, whose objects are
the bounded above, bounded below and bounded complexes respectively.

Our notation for distinguished triangles in $\cat{K}(\cat{M})$ or 
$\cat{D}(\cat{M})$ is either 
$L \xar{\alpha} M \xar{\beta} N \xar{\gamma} L[1]$,
or simply $L \to M \to N \distri$ 
if the names of the morphisms are not important.

A complex $P \in \cat{C}(\cat{M})$ is called {\em K-projective} if 
for any acyclic complex $N \in \cat{C}(\cat{M})$ the complex 
$\opn{Hom}_{\cat{M}}(P, N)$ is also acyclic. 
A complex $I \in \cat{C}(\cat{M})$ is called {\em K-injective} if 
for any acyclic complex $N \in \cat{C}(\cat{M})$ the complex 
$\opn{Hom}_{\cat{M}}(N, I)$ is also acyclic. 
These definitions were introduced in \cite{Sp}; in \cite[Section 3]{Ke} it is
shown that ``K-projective'' is the same as ``having property (P)'', and
``K-injective'' is the same as ``having property (I)''.

A K-projective resolution of 
$M \in \cat{C}(\cat{M})$ is a quasi-isomorphism $P \to M$ in 
$\cat{C}(\cat{M})$ with $P$ a K-projective complex. If every 
$M \in \cat{C}(\cat{M})$ admits some K-projective resolution, then we say that 
$\cat{C}(\cat{M})$ has enough K-projectives. 
Similarly for K-injectives.

Now we specialize to the case $\cat{M} := \cat{Mod} A$,
where $A$ is a ring. 
A complex $P \in \cat{C}(\cat{Mod} A)$ is called {\em K-flat} if 
for any acyclic complex $N \in \cat{C}(\cat{Mod} A)$ the complex 
$N \otimes_A P$ is also acyclic. Note that a K-projective
complex $P$ is K-flat. 

Here is a useful existence result.

\begin{prop} \label{prop:28}
Let $A$ be a ring, and let $M \in \cat{C}(\cat{Mod} A)$.
\begin{enumerate}
\item The complex $M$ admits a
quasi-isomorphism $P \to M$, where $P$ is a K-projective complex, and moreover
each component $P^i$ is a free $A$-module.

\item The complex $M$ admits a
quasi-isomorphism $P \to M$, where $P$ is a K-flat complex, and moreover
each component $P^i$ is a flat $A$-module.

\item The complex $M$ admits a quasi-isomorphism
$M \to I$, where $I$ is a K-injective complex, and moreover each component
$I^i$ is an injective \lb $A$-module. 
\end{enumerate}
\end{prop}

\begin{proof}
(1) This is proved in \cite[Subsection 3.1]{Ke}, when discussing the existence
of P-resolutions. Cf.\ \cite[Corollary 3.5]{Sp}.

\medskip \noindent
(2) This follows from (1), since any K-projective complex is also K-flat.

\medskip \noindent
(3) See \cite[Subsection 3.2]{Ke}. Cf.\ \cite[Proposition 3.11]{Sp}.
\end{proof}

In particular, the proposition says that $\cat{C}(\cat{Mod} A)$ has enough 
K-projectives, K-flats and K-injectives.

\begin{rem}
Let $(X, \mcal{A})$ be a ringed space, and let 
$\cat{Mod} \mcal{A}$ be the
category of sheaves of $\mcal{A}$-modules. 
It is known that $\cat{C}(\cat{Mod} \mcal{A})$ has enough K-injectives and
enough K-flats; but their structure is more complicated than in the case of \lb 
$\cat{C}(\cat{Mod} A)$, and Proposition \ref{prop:28} might not hold.
\end{rem}

Here are a few facts about K-projective and K-injective resolutions,
compiled from \cite{Sp, BN, Ke}. The first are: a bounded above complex of
projectives is K-projective, a bounded above complex of
flats is K-flat, and a bounded below complex of injectives is
K-injective.

Once again $\cat{M}$ is an abelian category. 
Let $\cat{E}$ be some triangulated category, and let 
$F : \cat{K}(\cat{M}) \to \cat{E}$ be a triangulated functor.
If $\cat{C}(\cat{M})$ has enough K-projectives, then the
left derived functor 
$(\mrm{L} F, \xi) : \cat{D}(\cat{M}) \to \cat{E}$ exists, and it is calculated
by K-projective resolutions. Likewise, if 
$\cat{K}(\cat{M})$ has enough K-injectives, then the right derived
functor $(\mrm{R} F, \xi) : \cat{D}(\cat{M}) \to \cat{E}$ exists,
and it is calculated by K-injective resolutions.

Let $M = \{ M^i \}_{i \in \Z}$ be a graded object of $\cat{M}$.
We define 
\begin{equation} \label{eqn:285}
\inf (M) := \opn{inf}\, \{ i \mid M^i \neq 0 \} \in \Z \cup \{ \pm \infty\}
\end{equation}
and 
\begin{equation} \label{eqn:286}
\sup (M) := \opn{sup}\, \{ i \mid  M^i \neq 0 \} \in \Z \cup \{ \pm \infty\} .
\end{equation}
The amplitude of $M$ is
\begin{equation} \label{eqn:287}
\opn{amp} (M) := \sup (M) - \inf (M) \in \N \cup \{ \pm \infty\} .
\end{equation}
(For $M = 0$ this reads $\inf (M) = \infty$, $\sup (M) = -\infty$ 
and $\opn{amp} (M) = -\infty$.)
Thus $M$ is bounded iff $\opn{amp} (M) < \infty$. 

For $M \in \cat{D}(\cat{M})$ we write
$\mrm{H} (M) := \{ \mrm{H}^i (M) \}_{i \in \Z}$.

\begin{dfn}
Let $\cat{M}$ and $\cat{M}'$ be abelian categories, and let
$F : \cat{D}(\cat{M}) \to \cat{D}(\cat{M}')$ be a triangulated functor.
Let $\cat{E} \subset \cat{D}(\cat{M})$ be a full additive subcategory (not
necessarily triangulated), and consider the restricted functor 
$F|_{\cat{E}} : \cat{E} \to \cat{D}(\cat{M}')$.
\begin{enumerate}
\item We say that $F|_{\cat{E}}$ has {\em finite cohomological dimension} if 
there exist some $n \in \N$ and $s \in \Z$ such that 
for every complex $M \in \cat{E}$ one has
\[ \opn{sup} \bigl( \mrm{H} (F (M)) \bigr) \leq 
\opn{sup} \bigl( \mrm{H} (M) \bigr)  + s \]
and
\[ \opn{inf} \bigl( \mrm{H} (F (M)) \bigr) \geq 
\opn{inf} \bigl( \mrm{H} (M) \bigr) + s - n .
\]
The smallest such number $n$ is called the {\em cohomological dimension} of 
$F|_{\cat{E}}$. 

\item If no such $n$ and $s$ exist then we say $F|_{\cat{E}}$ has {\em
infinite cohomological dimension}. 
\end{enumerate}
\end{dfn}

The number $s$ appearing in the definition represents the shift. (An easy
calculation shows that if $F|_{\cat{E}}$ is nonzero and has finite cohomological
dimension $n$, then the shift $s$ in the definition is unique.) 

If the functor $F$ has finite cohomological dimension, then it is a {\em
way-out functor in both directions}, in the sense of \cite[Section I.7]{RD}.
We will use this fact several times. 

\begin{exa}
Take a nonzero ring $A$, and let 
$P := A[1] \oplus A[2]$, a complex with zero differential concentrated in
degrees $-1$ and $-2$. The functor $F := P \otimes_A -$ 
has cohomological dimension $n = 1$, 
with shift $s = -1$.
\end{exa}

\begin{prop} \label{prop:21}
Let $\cat{M}$, $\cat{M}'$ and $\cat{M}''$ be abelian categories, and let
$F : \cat{D}(\cat{M}) \to \cat{D}(\cat{M}')$ and 
$F' : \cat{D}(\cat{M}') \to \cat{D}(\cat{M}'')$
be triangulated functors. Assume the cohomological dimensions of 
$F$ and $F'$ are $n$ and $n'$ respectively. Then the cohomological dimension
of $F' \circ F$ is at most $n + n'$.
\end{prop}

We leave out the easy proof. 

Here is a useful criterion for quasi-isomorphisms
(a variant of the way-out argument).
For $i, j \in \Z$ let 
$\cat{C}^{[i, j]}(\cat{M})$ be the full subcategory of 
$\cat{C}(\cat{M})$ whose objects are the complexes concentrated in the degree
range $[i, j] := \{ i, i+1, \ldots, j \}$.

\begin{prop} \label{prop:1.1}
Let $\cat{M}$ and $\cat{M}'$ be abelian categories, let
$F, G : \cat{M} \to \cat{C}(\cat{M}')$
be additive functors, and let $\eta : F \to G$ be a natural transformation. 
Assume $\cat{M}'$ has countable direct sums, and consider the extensions 
$F, G : \cat{C}(\cat{M}) \to \cat{C}(\cat{M}')$
by the direct sum totalization. Suppose $M \in \cat{C}(\cat{M})$ satisfies these
two conditions\tup{:}
\begin{enumerate}
\rmitem{i} There are $j_0, j_1 \in \Z$ such that 
$F(M^i), G(M^i) \in \cat{C}^{[j_0, j_1]}(\cat{M}')$
for every $i \in \Z$. 

\rmitem{ii} The homomorphism $\eta_{M^i} : F(M^i) \to G(M^i)$
is a quasi-isomorphism for every $i \in \Z$.
\end{enumerate}
Then $\eta_{M} : F(M) \to G(M)$ is a quasi-isomorphism.
\end{prop}

\begin{proof}
Step 1. Assume that $M$ is bounded. We prove that $\eta_{M}$ is a
quasi-iso\-morphism by induction on $\opn{amp}(M)$.
If $\opn{amp}(M) = 0$ then this is given. 
The inductive step is done using the stupid truncation functors 
\begin{equation}
\opn{stt}^{> i}(M), \opn{stt}^{\leq i}(M) : 
\cat{C}(\cat{M}) \to \cat{C}(\cat{M}) ,
\end{equation}
and the related short exact
sequences. See \cite[pages 69-70]{RD}, where the truncations
$\opn{stt}^{> i}(M)$ and $\opn{stt}^{\leq i}(M)$
 are denoted by 
$\tau_{> i}(M)$ and $\tau_{\leq i}(M)$ respectively.

\medskip \noindent
Step 2. Now $M$ is arbitrary. We have to prove that 
$\mrm{H}^i(\eta_{M}) : \mrm{H}^i (F(M)) \to \mrm{H}^i (G(M))$
is an isomorphism for every $i \in \Z$. 
For any $i \leq j$ there is the double truncation functor
$\opn{stt}^{[i, j]} := \opn{stt}^{\leq j} \circ \opn{stt}^{> i}$.
So let us fix $i$. 
The homomorphism $\mrm{H}^i(\eta_{M})$ in $\cat{M}'$ only depends on the
homomorphism of complexes
\[ \opn{stt}^{[i-1, i+1]}(\eta_{M}) :
\opn{stt}^{[i-1, i+1]} (F(M)) \to \opn{stt}^{[i-1, i+1]} (G(M)) . \]
Therefore we can replace $\eta_{M}$ with 
$\eta_{M'} : F(M') \to G(M')$,
where
\[ M' := \opn{stt}^{[j_0 + i - 1, j_1 + i +1]} (M) . \]
But $M'$ is bounded, so by part (1) the homomorphism $\eta_{M'}$ is a
quasi-isomorphism.  
\end{proof}

To end this section, here is a basic result we need, that we could not locate in
the
literature (but that was used implicitly in \cite{Sc}). 

\begin{prop} \label{prop:290}
Let $\cat{M}$ and
$\cat{N}$ be abelian categories, let 
$F : \cat{M} \to \cat{N}$ be an exact additive covariant functor, and let 
$G : \cat{M} \to \cat{N}$ be an exact additive contravariant functor. Then 
for any $M \in \cat{C}(\cat{M})$
there are isomorphisms 
$\mrm{H}^k(F(M)) \cong F(\mrm{H}^k(M))$ and 
$\mrm{H}^{-k}(G(M)) \cong G(\mrm{H}^k(M))$ in $\cat{N}$. Moreover, these
isomorphisms are functorial in $M$, $F$ and $G$. 
\end{prop}

\begin{proof}
These are very degenerate cases of Grothendieck spectral sequences.
Here is a direct proof. We use the notation 
$\mrm{Z}^k(M) := \opn{Ker}(\d : M^k \to M^{k+1})$,
which is the object of $k$-cocycles of $M$, and  
$\mrm{Y}^k(M) := \opn{Coker}(\d : M^{k-1} \to M^{k})$, 
which does not have a name. There are functorial isomorphisms
\[ \mrm{H}^k(M) \cong \opn{Coker} \bigl( \d : M^{k-1} \to \mrm{Z}^k(M) \bigr)
\cong \opn{Ker} \bigl( \d : \mrm{Y}^k(M) \to  M^{k+1}  \bigr) . \]

For any additive functor $F : \cat{M} \to \cat{N}$ (not necessarily exact) there
is an obvious morphism 
$\al : F(\mrm{Z}^k(M)) \to \mrm{Z}^k(F(M))$, and it induces 
a morphism 
$\bar{\al} : F(\mrm{H}^k(M)) \to \mrm{H}^k(F(M))$.
An easy calculation shows that when $F$ is exact, the morphisms $\al$ and
$\bar{\al}$ are isomorphisms.

Given a contravariant additive functor $G : \cat{M} \to \cat{N}$, 
there is a morphism (slightly less obvious than $\al$, because of the change in
direction) 
$\be : G(\mrm{Y}^k(M)) \to \lb \mrm{Z}^{-k}(G(M))$. This  induces a morphism 
$\bar{\be} : G(\mrm{H}^k(M)) \to \mrm{H}^{-k}(G(M))$.
When $G$ is exact, the morphisms $\be$ and $\bar{\be}$ are isomorphisms.
\end{proof}

\begin{cor} \label{cor:293}
Let $A$ be a ring, $M$ a complex of $A$-modules, $P$ a flat $A$-module, and
$I$ and injective $A$-module. There are isomorphisms
\[ \mrm{H}^k(M \ot_A P) \cong \mrm{H}^k(M) \ot_A P \]
and 
\[ \mrm{H}^{-k} ( \opn{Hom}_A(M, I) ) \cong  \opn{Hom}_A( \mrm{H}^k(M), I) , \]
functorial in $M, P$ and $I$.
\end{cor}

\begin{proof}
Take $F(M) := M \ot_A P$ and 
$G(M) := \opn{Hom}_A(M, I)$, and use the proposition above.
(These are degenerate cases of the K\"unneth Theorems.)
\end{proof}

\section{The Derived Completion and Torsion Functors} \label{sec:comp}

In this section $A$ is a commutative ring, and $\a$ is an ideal in
it. We do not assume that $\a$ is finitely generated or that $A$ is 
$\a$-adically complete. 

For any $i \in \N$ let
$A_i := A / \a^{i+1}$.
The collection of rings $\{ A_i \}_{i \in \N}$ forms an inverse system. 
Following \cite{GM, AJL1}, for an $A$-module $M$ we write
\begin{equation}
\Lambda_{\a} (M) := \lim_{\leftarrow i}\, (A_i \otimes_A M) 
\end{equation}
for the $\a$-adic completion of $M$, although we sometimes use the more
conventional (yet possibly ambiguous) notation $\what{M}$. 
We get an additive functor 
$\Lambda_{\a} : \cat{Mod} A \to \cat{Mod} A$.
Recall that there is a functorial homomorphism
\begin{equation} \label{eqn:207}
\tau_M : M \to \Lambda_{\a} (M)
\end{equation}
for $M \in \cat{Mod} A$, coming from the homomorphisms
$M \to A_i \otimes_A M$. The module $M$ is called {\em $\a$-adically complete}
if $\tau_M$ is an isomorphism. (Some texts, such as \cite{Bo}, would say that
$M$ is separated and complete). As customary, when $M$ is
complete we usually identify $M$ with $\Lambda_{\a} (M)$ via $\tau_M$.

If the ideal $\a$ is finitely generated, then the functor $\Lambda_{\a}$ is
idempotent, in the sense that the homomorphism
\[ \tau_{\Lambda_{\a} (M)} : \Lambda_{\a} (M) \to 
\Lambda_{\a} (\Lambda_{\a} (M)) \]
is an isomorphism for every module $M$ (see \cite[Corollary 3.6]{Ye2}). 

Let $\what{A} :=  \Lambda_{\a} (A)$. Then $\what{A}$ is a 
ring, and $\tau_A : A \to \what{A}$ is a ring homomorphism. 
If $A$ is noetherian then $\what{A}$ is also noetherian, and flat over $A$. 
One can view the completion as a functor 
$\Lambda_{\a} : \cat{Mod} A \to \cat{Mod} \what{A}$. 
But in this paper we shall usually ignore this. 

\begin{rem}
The full subcategory of $\cat{Mod} A$ consisting of
$\a$-adically complete modules is additive, but not abelian in general. 

It is well known that when $A$ is noetherian, the completion functor
$\Lambda_{\a}$ is exact on 
$\cat{Mod}_{\mrm{f}} A$, the category of finitely generated modules. 
However, on $\cat{Mod} A$ the functor $\Lambda_{\a}$ is neither left exact nor
right exact, even in the noetherian case (see \cite[Examples 3.19 and
3.20]{Ye2}). 

When $A$ is not noetherian, we do not know if $\what{A}$ is flat over $A$.
Still, if $\a$ is finitely generated, and we let 
$\what{\a} := \what{A} \a \subset \what{A}$, then $\what{A}$ is 
$\what{\a}$-adically complete; this follows from \cite[Corollary 3.6]{Ye2}.

If the ideal $\a$ is not finitely generated, things are even worse: the 
functor $\Lambda_{\a}$ can fail to be idempotent; i.e.\ the completion
$\Lambda_{\a} (M)$ of a module $M$ could fail to be complete. See 
\cite[Example 1.8]{Ye2}.
\end{rem}

As for any additive functor, the functor $\Lambda_{\a}$ has a left derived
functor
\begin{equation} \label{eqn:51}
\mrm{L} \Lambda_{\a} : \cat{D}(\cat{Mod} A) \to \cat{D}(\cat{Mod} A) \ , \
\xi : \mrm{L} \Lambda_{\a} \to \Lambda_{\a}
\end{equation}
constructed using K-projective resolutions.

The next result was proved in \cite{AJL1}. Since this is so fundamental, we
chose to reproduce the easy proof.

\begin{lem}[\cite{AJL1}] \label{lem:43}
Let $P$ be an acyclic K-flat complex of $A$-modules. Then the complex 
$\Lambda_{\a} (P)$ is also acyclic. 
\end{lem}

\begin{proof}
Since $P$ is both acyclic and K-flat, for any $i$ we have an acyclic complex 
$A_i \otimes_A P$. The collection of complexes 
$\{ A_i \otimes_A P \}_{i \in \N}$ is an inverse system, and the homomorphism
$A_{i+1} \otimes_A P^j \to A_{i} \otimes_A P^j$
is surjective for every $i$ and $j$. But
$\Lambda_{\a} (P^j) = \lim_{\leftarrow i} \, (A_{i} \otimes_A P^j)$.
By the Mittag-Leffler argument (see \cite[Proposition 1.12.4]{KS1} or
\cite[Theorem 3.5.8]{We})
the complex $\Lambda_{\a} (P)$ is acyclic.
\end{proof}

\begin{prop} \label{prop:26}
If $P$ is a K-flat complex then the morphism
$\xi_P : \mrm{L} \Lambda_{\a} (P) \to \Lambda_{\a} (P)$
in $\cat{D}(\cat{Mod} A)$
is an isomorphism. Thus we can calculate $\mrm{L} \Lambda_{\a}$ using K-flat
resolutions.
\end{prop}

\begin{proof}
This is immediate from Lemma \ref{lem:43}; cf.\ \cite[Theorem I.5.1]{RD}.
\end{proof}

\begin{prop}[\cite{AJL1}] \label{prop:comp.1}
Let $M \in \cat{D}(\cat{Mod} A)$. There is a morphism
$\tau^{\mrm{L}}_M : M \to \mrm{L} \Lambda_{\a} (M)$
in $\cat{D}(\cat{Mod} A)$, functorial in $M$, such that 
$\xi_M \circ \tau^{\mrm{L}}_M = \tau_M$
as morphisms $M \to \Lambda_{\a} (M)$. 
\end{prop}

\begin{proof}
Given $M \in \cat{D}(\cat{Mod} A)$ let us choose a K-projective resolution
$\phi : P \to M$. Since $\phi$ and $\xi_P$ are isomorphisms in
$\cat{D}(\cat{Mod} A)$, we can define 
\[ \tau^{\mrm{L}}_M := \mrm{L} \Lambda_{\a}(\phi) \circ \xi_P^{-1}
\circ \tau_P \circ \phi^{-1} : M \to \mrm{L} \Lambda_{\a} (M)  . \]
This is independent of the the chosen resolution $\phi$, and satisfies 
$\xi_M \circ \tau_M = \tau^{\mrm{L}}_M$.
\end{proof}

\begin{dfn} \label{dfn:2} 
\begin{enumerate}
\item A complex $M \in \cat{D}(\cat{Mod} A)$ is called {\em
$\a$-adically cohomologically complete}
if the morphism $\tau^{\mrm{L}}_M : M \to \mrm{L} \Lambda_{\a} (M)$ is an
isomorphism. 

\item The full subcategory of $\cat{D}(\cat{Mod} A)$ consisting of 
$\a$-adically cohomologically complete complexes is denoted by 
$\cat{D}(\cat{Mod} A)_{\a \tup{-com}}$.
\end{enumerate}
\end{dfn}

It is clear that the subcategory 
$\cat{D}(\cat{Mod} A)_{\a \text{-} \tup{com}}$ is triangulated.

The notion of cohomologically complete complex is quite illusive. See Example
\ref{exa:1} . 

For an $A$-module $M$ and $i \in \N$ we identify 
$\opn{Hom}_A(A_i, M)$ with the submodule 
\[ \{ m \in M \mid \a^{i+1} m = 0 \} \subset M .  \]
The {\em $\a$-torsion submodule} of $M$ is
\[ \Gamma_{\a} (M) := \bigcup_{i \in \N} \, \opn{Hom}_A(A_i, M) 
\subset M . \]
The module $M$ is called an {\em $\a$-torsion module} if
$\Gamma_{\a} (M) = M$. We denote by 
$\cat{Mod}_{\a \tup{-tor}} A$ the full subcategory of $\cat{Mod} A$ consisting
of $\a$-torsion modules.

We get an additive functor 
$\Gamma_{\a} : \cat{Mod} A \to \cat{Mod} A$.
In fact this is a left exact functor. 
There is a functorial homomorphism 
$\sigma_M : \Gamma_{\a} (M) \to M$
which is just the inclusion. The functor $\Gamma_{\a}$ is idempotent, in the
sense that  
$\sigma_{\Gamma_{\a} (M)} : \Gamma_{\a} (\Gamma_{\a} (M)) \to  \Gamma_{\a} (M)$
is bijective. 

Like every additive functor, the functor $\Gamma_{\a}$ has a right derived
functor 
\begin{equation} \label{eqn:53}
\mrm{R} \Gamma_{\a} : \cat{D}(\cat{Mod} A) \to \cat{D}(\cat{Mod} A) \ , \
\xi : \Gamma_{\a} \to \mrm{R} \Gamma_{\a}
\end{equation}
constructed using K-injective resolutions. 

\begin{prop}
There is a functorial morphism
$\sigma^{\mrm{R}}_M : \mrm{R} \Gamma_{\a} (M) \to M $,
such that 
$\sigma_M = \sigma^{\mrm{R}}_M \circ \xi_M$
as morphisms $\Gamma_{\a} (M) \to M$ in $\cat{D}(\cat{Mod} A)$. 
\end{prop}

\begin{proof}
Choose a K-injective resolution $\phi : M \to I$, and define 
\[ \sigma^{\mrm{R}}_M := \phi^{-1} \circ \sigma_I \circ 
\xi_I^{-1} \circ \mrm{R} \Gamma_{\a}(\phi) . \]
This is independent of the resolution. 
\end{proof}

\begin{dfn} \label{dfn:8} 
\begin{enumerate}
\item A complex $M \in \cat{D}(\cat{Mod} A)$ is called {\em
cohomologically $\a$-torsion}
if the morphism $\sigma^{\mrm{R}}_M : \mrm{R} \Gamma_{\a} (M) \to M$ is an
isomorphism. 

\item The full subcategory of $\cat{D}(\cat{Mod} A)$ consisting of 
cohomologically $\a$-torsion complexes is denoted by 
$\cat{D}(\cat{Mod} A)_{\a \text{-tor}}$.

\item 
We denote by 
$\cat{D}_{\a \tup{-tor}}(\cat{Mod} A)$ the full subcategory of
$\cat{D}(\cat{Mod} A)$ consisting of the complexes whose cohomology modules are
in $\cat{Mod}_{\a \tup{-tor}} A$.
\end{enumerate}
\end{dfn}

It is clear that the subcategory $\cat{D}(\cat{Mod} A)_{\a \text{-} \mrm{tor}}$
is triangulated.

Since $\cat{Mod}_{\a \tup{-tor}} A$ is a thick abelian subcategory of 
$\cat{Mod} A$,  it follows
that \lb $\cat{D}_{\a \tup{-tor}}(\cat{Mod} A)$ is a triangulated category.
Note that 
$\Gamma_{\a} (I) \in \cat{D}_{\a \tup{-tor}}(\cat{Mod} A)$
for any K-injective complex $I$. Therefore  
\begin{equation} \label{eqn:54}
\cat{D}(\cat{Mod} A)_{\a \text{-} \mrm{tors}} \subset 
\cat{D}_{\a \tup{-tor}}(\cat{Mod} A) . 
\end{equation}
Later (in Corollary \ref{cor:17}) we shall see that there is equality in 
(\ref{eqn:54}) under some extra assumption.

\section{Koszul Complexes and Weak Proregularity} \label{sec:kosz}

In this section we define {\em weakly proregular sequences}. We also set up
notation to be used later. 
The definitions and most of the results in this section are  contained in
\cite{AJL1} and \cite{Sc}. We have included our own short proofs, for the
benefit of the reader. We also give a new motivating example at the end.

Let $A$ be a commutative ring (not necessarily noetherian). 
Recall that for an element $a \in A$ the {\em Koszul complex} 
$\opn{K}(A; a)$ is the complex 
\begin{equation}
\mrm{K}(A; a) :=  \bigl( \cdots \to 0 \to A \xar{a \cdot} A \to 0 \to
\cdots \bigr) 
\end{equation}
concentrated in degrees $-1$ and $0$. 
Now let $\bsym{a} = (a_1, \ldots, a_n)$ be a sequence of elements of $A$. 
The Koszul complex associated to
$\bsym{a}$ is the complex of $A$-modules
\begin{equation}
\opn{K}(A; \bsym{a}) := 
\opn{K}(A; a_1) \otimes_A \cdots \otimes_A \opn{K}(A; a_n) .
\end{equation}
Observe that $\opn{K}(A; \bsym{a})^{0} \cong A$, and
$\opn{K}(A; \bsym{a})^{-1}$ is a free $A$-module of rank $n$. 
Moreover,  $\opn{K}(A; \bsym{a})$ is a super-commutative DG algebra:
as a graded algebra it is the exterior algebra over $A$
of the module $\opn{K}(A; \bsym{a})^{-1}$.
There is a DG algebra homomorphism 
\begin{equation} \label{eqn:251}
e_{\bsym{a}} : A \to \opn{K}(A; \bsym{a}) . 
\end{equation}

Let us denote by $(\bsym{a})$ the ideal generated by the sequence $\bsym{a}$,
so that 
\[  A / (\bsym{a}) \cong 
A / (a_1) \otimes_A \cdots \otimes_A A / (a_n)  \]
as $A$-algebras.
There is an $A$-algebra isomorphism
\begin{equation} \label{eqn:217}
\mrm{H}^0 (\opn{K}(A; \bsym{a})) \cong  A / (\bsym{a}) . 
\end{equation}

For any $j \geq i$ in $\N$ there is a homomorphism of complexes 
\begin{equation}
p_{a, j, i}: \opn{K}(A; a^{j}) \to \opn{K}(A; a^{i}) , 
\end{equation}
which is the identity in degree $0$, and multiplication by $a^{j-i}$ in degree
$-1$. This operation makes sense also for sequences: given a sequence $\bsym{a}$
as above, let us write $\bsym{a}^i := (a_1^i, \ldots, a_n^i)$.
There is a homomorphism of complexes 
\begin{equation} \label{eqn:201}
p_{\bsym{a}, j, i} :  \opn{K}(A; \bsym{a}^{j}) \to \opn{K}(A; \bsym{a}^{i})
\ , \ p_{\bsym{a}, j, i} := p_{a_1, j, i} \ot \cdots \ot p_{a_n, j, i}  . 
\end{equation}
In fact $p_{\bsym{a}, j, i}$ is a homomorphism of DG algebras, and
$\mrm{H}^0 (p_{\bsym{a}, j, i})$ corresponds via (\ref{eqn:217}) to the
canonical surjection 
$A / (\bsym{a}^{j}) \to A / (\bsym{a}^i)$. 
The homomorphisms 
\begin{equation} \label{eqn:256}
\mrm{H}^k (p_{\bsym{a}, j, i}) :
\mrm{H}^k ( \opn{K}(A; \bsym{a}^{j}) ) \to 
\mrm{H}^k ( \opn{K}(A; \bsym{a}^{i}) )
\end{equation}
make 
$\bigl\{ \mrm{H}^k ( \opn{K}(A; \bsym{a}^{i}) ) \bigr\}_{i \in \N}$
into an inverse system of $A$-modules.

Let $P$ be a finite rank free $A$-module.
We shall often write $P^{\vee} := \opn{Hom}_A(P, A)$. 
Given any $A$-module $M$, there is an isomorphism 
\begin{equation} \label{eqn:300}
\opn{Hom}_A(P, M) \cong P^{\vee} \ot_A M ,
\end{equation}
functorial in $M$ and $P$.

The {\em dual Koszul complex} associated to the sequence 
$\bsym{a} = (a_1, \ldots, a_n)$ is the complex 
\begin{equation} \label{eqn:5.2}
\opn{K}^{\vee}(A; \bsym{a}) := 
\opn{Hom}_A \bigl( \opn{K}(A; \bsym{a}) , A \bigr) .
\end{equation}
This is complex of finite rank free $A$-modules, concentrated in degrees 
$0, \ldots, n$. Indeed, for a single element $a$ there is a canonical
isomorphism of complexes 
\begin{equation} \label{eqn:203}
\opn{K}^{\vee}(A; a) \cong 
\bigl( \cdots \to 0 \to A \xar{a \cdot } A \to 0 \to
\cdots \bigr) 
\end{equation}
with $A$ sitting in degrees $0$ and $1$. And for the sequence we have 
\[ \opn{K}^{\vee}(A; \bsym{a}) \cong 
\opn{K}^{\vee}(A; a_1) \otimes_A \cdots \otimes_A \opn{K}^{\vee}(A; a_n) . \]
The dual 
$e_{\bsym{a}}^{\vee} := \opn{Hom}(e_{\bsym{a}}, 1_A)$ 
of $e_{\bsym{a}}$ is a homomorphism of complexes 
\begin{equation} \label{eqn:252}
e_{\bsym{a}}^{\vee} : \opn{K}^{\vee}(A; \bsym{a}) \to A . 
\end{equation}

For any $j \geq i$ in $\N$ there is a homomorphism of complexes 
\begin{equation} \label{eqn:209}
p_{\bsym{a}, j, i}^{\vee} :
\opn{K}^{\vee}(A; \bsym{a}^{i}) \to \opn{K}^{\vee}(A; \bsym{a}^{j}) , 
\end{equation}
which comes from dualizing the homomorphism (\ref{eqn:201}). 
In this way the collection 
$\bigl\{ \opn{K}^{\vee}(A; \bsym{a}^{i}) \bigr\}_{i \in \N}$
becomes a direct system of complexes. 
The {\em infinite dual Koszul complex} associated to a sequence $\bsym{a}$
in $A$ is the complex of $A$-modules
\begin{equation}
\opn{K}^{\vee}_{\infty}(A; \bsym{a}) := 
\lim_{i \to}\, \opn{K}^{\vee}(A; \bsym{a}^{i})  . 
\end{equation}

For a single element $a \in A$ the infinite  dual Koszul complex looks like
this: there is a canonical isomorphism 
\begin{equation} \label{eqn:202}
\opn{K}^{\vee}_{\infty}(A; a) \cong 
\bigl( \cdots \to 0 \to A \xar{} A[a^{-1}] \to 0 \to
\cdots \bigr) 
\end{equation}
where $A$ is in degree $0$, $A[a^{-1}]$ is in degree $1$, and the differential 
$A \to A[a^{-1}]$ is the ring homomorphism. 
For a sequence we have 
\begin{equation} \label{eqn:210}
 \opn{K}^{\vee}_{\infty}(A; \bsym{a})  \cong 
\opn{K}^{\vee}_{\infty}(A; a_1) \ot_A \cdots \ot_A 
\opn{K}^{\vee}_{\infty}(A; a_n) .
\end{equation}
Thus $\opn{K}^{\vee}_{\infty}(A; \bsym{a})$ is a complex of flat
$A$-modules concentrated in degrees $0, \ldots, n$.

Let us write 
\begin{equation}
e_{\bsym{a}, i}^{\vee}  :
\opn{K}^{\vee}(A; \bsym{a}^i) \to A \ , \
e_{\bsym{a}, i}^{\vee} : = e_{\bsym{a}^i}^{\vee} \ ,
\end{equation}
where $e_{\bsym{a}^i}^{\vee}$ is from  (\ref{eqn:252}). 
The homomorphisms $e_{\bsym{a}, i}^{\vee}$ respect the direct system, and in the
limit we get 
\begin{equation} \label{eqn:213}
e_{\bsym{a}, \infty}^{\vee}  : 
\opn{K}^{\vee}_{\infty}(A; \bsym{a}) \to A \ , \
e_{\bsym{a}, \infty}^{\vee} := \lim_{i \to}\, e_{\bsym{a}, i}^{\vee} \ . 
\end{equation}

Let $\a$ be the ideal in $A$ generated by the sequence 
$\bsym{a} = (a_1, \ldots, a_n)$.
{}From equations (\ref{eqn:202}) and (\ref{eqn:210}) we see that 
\begin{equation} \label{eqn:280}
\mrm{H}^0 \bigl( \opn{K}^{\vee}_{\infty}(A; \bsym{a}) \otimes_A M \bigr)
\cong \Gamma_{\a} (M) 
\end{equation}
for any $M \in \cat{Mod} A$. 
This gives rise to a functorial homomorphism of complexes
\begin{equation} \label{eqn:4.10}
v_{\bsym{a}, M} : \Gamma_{\a}  (M) \to 
\opn{K}^{\vee}_{\infty}(A; \bsym{a}) \otimes_A M 
\end{equation}
that satisfies 
\begin{equation} \label{eqn:222}
(e_{\bsym{a}, \infty}^{\vee} \ot 1_M) \circ v_{\bsym{a}, M}= \sigma_M 
\end{equation}
as homomorphisms $\Gamma_{\a}  (M) \to M$.

An inverse system 
$\{ M_i \}_{i \in \N}$
of abelian groups, with transition maps $p_{j, i} :M_j \to M_i$,
is called {\em pro-zero} if for every $i$ there exists $j \geq i$ such that 
$p_{j, i}$ is zero. (This is the name used in \cite{Sc}.)
We shall use the fact that a pro-zero inverse system satisfies the
Mittag-Leffler condition. 
See \cite[Definition 3.5.6]{We}, where the condition ``pro-zero'' is called the 
``trivial Mittag-Leffler'' condition. 

\begin{dfn} \label{dfn:250} 
\begin{enumerate}
\item Let $\bsym{a}$ be a finite sequence in a ring $A$. The sequence 
$\bsym{a}$ is
called {\em a weakly proregular sequence}  if for every $k < 0$ the inverse
system 
$\bigl\{ \mrm{H}^k ( \opn{K}(A; \bsym{a}^{i}) ) \bigr\}_{i \in \N}$ 
(see (\ref{eqn:256})) is pro-zero.

\item An ideal $\a$ in a ring $A$ is called a {\em a weakly proregular ideal} if
it is generated by some weakly proregular sequence. 
\end{enumerate}
\end{dfn}

This definition comes from \cite[Correction]{AJL1}.
The etymology and history of related concepts are explained in \cite{AJL1} and
\cite{Sc}.  

\begin{exa}
A regular sequence $\bsym{a}$ is weakly proregular, since 
$\mrm{H}^k \bigl( \opn{K}(A; \bsym{a}^{i}) \bigr) = 0$ for all $i > 0$ and 
$k < 0$.
\end{exa}

\begin{lem}  \label{lem:250}
Let $\{ M_i \}_{i \in \N}$ be an inverse system of $A$-modules. The
following conditions are equivalent\tup{:}
\begin{enumerate}
 \rmitem{i} The system $\{ M_i \}_{i \in \N}$ is pro-zero.

\rmitem{ii} For every injective $A$-module $I$,
$\lim_{i \to}\, \opn{Hom}_A(M_i, I) = 0$. 
\end{enumerate}
\end{lem}

\begin{proof}
The implication (i) $\Rightarrow$ (ii) is trivial. 
For the other direction, take any $i \in \N$, and choose an embedding 
$\phi : M_i \inj I$ for some injective module $I$. So $\phi$ is an
element of $\opn{Hom}_A(M_i, I)$. Since the limit is zero, there is some $j \geq
i$ such that $\phi \circ p_{j, i} = 0$.  Here $p_{j, i} :M_j \to M_i$ is the
transition map. This implies that $p_{j, i} = 0$.
\end{proof}

\begin{thm} \label{thm:252}
Let $\bsym{a}$ be a finite sequence in a ring $A$. The following conditions are
equivalent\tup{:}
\begin{enumerate}
\rmitem{i} The sequence $\bsym{a}$ is weakly proregular.

\rmitem{ii} For any injective module $I$ and any $k > 0$ the $A$-module 
$\mrm{H}^k (\opn{K}^{\vee}_{\infty}(A; \bsym{a}) \otimes_A I)$
is zero. 
\end{enumerate}
\end{thm}

\begin{proof}
Take any injective $A$-module $I$. We get isomorphisms:
\[ \begin{aligned}
& 
\mrm{H}^k \bigl (\opn{K}^{\vee}_{\infty}(A; \bsym{a}) \otimes_A I \bigr)
\cong^{\lozenge} \mrm{H}^k \bigl ( \lim_{j \to}\, \bigl( \opn{K}^{\vee}(A;
\bsym{a}^j)
\otimes_A I \bigr) \bigr)
\\
& \qquad 
\cong^{\lozenge} \lim_{j \to}\,  \mrm{H}^k \bigl ( \opn{K}^{\vee}(A; \bsym{a}^j)
\otimes_A
I \bigr) 
\cong^{\vartriangle} \lim_{j \to}\, \mrm{H}^k \bigl( \opn{Hom}_A
\bigl ( \opn{K}(A; \bsym{a}^j), I \bigr) \bigr)
\\
& \qquad 
\cong^{\heartsuit} \lim_{j \to}\, \opn{Hom}_A \bigl ( 
\mrm{H}^{-k}  \bigl(  \opn{K}(A; \bsym{a}^j) \bigr), I \bigr) .
\end{aligned} \]
The isomorphisms marked $\lozenge$ are because  direct
limits commute with tensor products and cohomology; 
the isomorphism  $\vartriangle$ is by (\ref{eqn:300}); 
and the isomorphism marked $\heartsuit$ is due to Corollary \ref{cor:293}. By
Lemma \ref{lem:250} the vanishing of this last limit 
for every $k > 0$ is equivalent to weak proregularity.
\end{proof}

\begin{cor} \label{cor:250}
Let $\bsym{a}$ be a weakly proregular sequence in $A$, $\a$ the ideal generated
by $\bsym{a}$, and $I$ a K-injective complex in $\cat{C}(\cat{Mod} A)$.
Then the homomorphism
\[ v_{\bsym{a}, I} : \Gamma_{\a} (I) \to 
\opn{K}^{\vee}_{\infty}(A; \bsym{a}) \otimes_A I  \]
is a quasi-isomorphism.
\end{cor}

\begin{proof}
By Proposition \ref{prop:28}(2) we can find a quasi-isomorphism
$I \to J$, where $J$ is K-injective and every $A$-module $J^i$ is injective.
Consider the commutative diagram 
\[ \UseTips \xymatrix @C=7ex @R=5ex {
\Gamma_{\a} (I)
\ar[r]^(0.35){v_{\bsym{a}, I}}
\ar[d]
& 
\opn{K}^{\vee}_{\infty}(A; \bsym{a}) \otimes_A I
\ar[d]
\\
\Gamma_{\a} (J)
\ar[r]^(0.35){v_{\bsym{a}, J}}
& 
\opn{K}^{\vee}_{\infty}(A; \bsym{a}) \otimes_A J
} \]
in $\cat{C}(\cat{Mod} A)$.
The vertical arrows are quasi-isomorphisms (for instance because $I \to J$ is a
homotopy equivalence). It suffices to prove that 
$v_{\bsym{a}, J}$ is a quasi-isomorphism. 

Let us write
$F(M) := \Gamma_{\a} (M)$ and 
$G(M) := \opn{K}^{\vee}_{\infty}(A; \bsym{a}) \otimes_A M$
for $M \in \cat{Mod} A$. We need to show that 
$v_{\bsym{a}, J} : F(J) \to G(J)$ is a quasi-isomorphism.
By Proposition \ref{prop:1.1} we may assume that 
$J$ is a single injective module. In this case we know that 
$\mrm{H}^0(v_{\bsym{a}, J})$ is bijective; see (\ref{eqn:280}). Theorem
\ref{thm:252} implies that $\mrm{H}^k(v_{\bsym{a}, J})$ is bijective for 
$k > 0$. And of course 
\[ \mrm{H}^k \bigl( \Gamma_{\a} (J) \bigr) = 
\mrm{H}^k \bigl(  \opn{K}^{\vee}_{\infty}(A; \bsym{a}) \otimes_A J \bigr)= 0 \]
for all $k < 0$. 
Hence $v_{\bsym{a}, J}$ is a quasi-isomorphism. 
\end{proof}

\begin{cor} \label{cor:kosz.1}
Let $\bsym{a}$ be a weakly proregular sequence in $A$, and $\a$ the ideal
generated by $\bsym{a}$.
For any $M \in \cat{D}(\cat{Mod} A)$ there is an isomorphism
\[ v^{\mrm{R}}_{\bsym{a}, M} : \mrm{R} \Gamma_{\a} (M) \to 
\opn{K}^{\vee}_{\infty}(A; \bsym{a}) \otimes_A M \]
in $\cat{D}(\cat{Mod} A)$. The isomorphism $v^{\mrm{R}}_{\bsym{a}, M}$ is
functorial in $M$, and satisfies 
\[ (e_{\bsym{a}, \infty}^{\vee} \ot 1_M) \circ v^{\mrm{R}}_{\bsym{a}, M}=
\sigma^{\mrm{R}}_M  \]
as morphisms $\mrm{R} \Gamma_{\a}  (M) \to M$.
\end{cor}

\begin{proof}
It is enough to consider a K-injective complex $M = I$. 
We define $v^{\mrm{R}}_{\bsym{a}, I} := v_{\bsym{a}, I}$ as in (\ref{eqn:4.10}).
Due to equation (\ref{eqn:222}) the morphism $v^{\mrm{R}}_{\bsym{a}, I}$
satisfies the parallel derived equation.  
By Corollary \ref{cor:250} the morphism $v^{\mrm{R}}_{\bsym{a}, I}$ is an
isomorphism in $\cat{D}(\cat{Mod} A)$.
\end{proof}

The corollary says that the diagram
\begin{equation}
\UseTips \xymatrix @C=7ex @R=6ex {
\mrm{R} \Gamma_{\a} (M)
\ar[r]^(0.4){v_M^{\mrm{R}}}
\ar[dr]_{\sigma^{\mrm{R}}_M}
& 
\opn{K}^{\vee}_{\infty}(A; \bsym{a}) \otimes_A M
\ar[d]^{e_{\bsym{a}, \infty}^{\vee} \ot 1_M}
\\
&
M
} 
\end{equation}
 in $\cat{D}(\cat{Mod} A)$ is commutative.

\begin{cor} \label{cor:260}
Let $\a$ be a weakly proregular ideal in A. 
Then the functor $\mrm{R} \Gamma_{\a}$ has finite cohomological dimension.
More precisely, if $\a$ can be generated by a weakly proregular sequence of
length $n$, then the cohomological dimension of $\mrm{R} \Gamma_{\a}$ is at most
$n$. 
\end{cor}

\begin{proof}
Choose any generating sequence 
$\bsym{a} = (a_1, \ldots, a_n)$ for $\a$. 
By Corollary \ref{cor:kosz.1} there is an isomorphism 
$\mrm{R} \Gamma_{\a} (M) \cong \opn{K}^{\vee}_{\infty}(A; \bsym{a}) \otimes_A M$
for any $M \in \cat{D}(\cat{Mod} A)$. But the amplitude of the complex 
$\opn{K}^{\vee}_{\infty}(A; \bsym{a})$ is $n$ (if $A$ is nonzero).
\end{proof}

\begin{lem} \label{lem:260}
For a finite sequence $\bsym{a}$ of elements of $A$, the
homomorphisms
\[ e_{\bsym{a}, \infty}^{\vee} \ot 1 , \, 
1 \otimes e_{\bsym{a}, \infty}^{\vee} :
\opn{K}^{\vee}_{\infty}(A; \bsym{a}) \otimes_A \opn{K}^{\vee}_{\infty}(A;
\bsym{a}) 
\to \opn{K}^{\vee}_{\infty}(A; \bsym{a}) \]
are quasi-isomorphisms.
\end{lem}

\begin{proof}
By symmetry it is enough to look only at 
\[ 1 \otimes e_{\bsym{a}, \infty}^{\vee} : 
\opn{K}^{\vee}_{\infty}(A; \bsym{a}) \otimes_A 
\opn{K}^{\vee}_{\infty}(A; \bsym{a})  \to 
\opn{K}^{\vee}_{\infty}(A; \bsym{a})  . \]
Write $\bsym{a} = (a_1, \ldots, a_n)$. 
Since 
$e_{\bsym{a}, \infty}^{\vee} = e_{a_1, \infty}^{\vee} \ot \cdots \ot 
e_{a_n, \infty}^{\vee}$,
and since the complexes 
$\opn{K}^{\vee}_{\infty}(A; a_i)$ are K-flat, 
it is enough to consider the case $n = 1$ and $a = a_1$. 
Here we have a surjective homomorphism of complexes 
\[ 1 \otimes e_{a, \infty}^{\vee} : 
\opn{K}^{\vee}_{\infty}(A; a) \otimes_A 
\opn{K}^{\vee}_{\infty}(A; a)  \to 
\opn{K}^{\vee}_{\infty}(A; a)  . \]
The kernel is the complex 
$ A[a^{-1}] \xar{\d} A[a^{-1}]$,
concentrated in degrees $1, 2$; and it is acyclic. 
\end{proof}

\begin{cor} \label{cor:11}
Let $\a$ be a weakly proregular ideal in a ring $A$.
For any $M \in \cat{D}(\cat{Mod} A)$ the morphism 
\[ \sigma^{\mrm{R}}_{\mrm{R} \Gamma_{\a} (M)} : 
\mrm{R} \Gamma_{\a} (\mrm{R} \Gamma_{\a} (M)) \to \mrm{R} \Gamma_{\a}  (M) \]
is an isomorphism. Thus the functor 
\[ \mrm{R} \Gamma_{\a} : \cat{D}(\cat{Mod} A) \to \cat{D}(\cat{Mod} A) \]
is idempotent.
\end{cor}

\begin{proof}
By Corollary \ref{cor:kosz.1} we can replace 
$\sigma^{\mrm{R}}_{\mrm{R} \Gamma_{\a} (M)}$
with 
\[ e_{\bsym{a}, \infty}^{\vee} \ot 1_{\opn{K}^{\vee}_{\infty}} \ot 1_M :
\opn{K}^{\vee}_{\infty}(A; \bsym{a}) \ot_A \opn{K}^{\vee}_{\infty}(A; \bsym{a})
\ot_A M \to 
\opn{K}^{\vee}_{\infty}(A; \bsym{a}) \ot_A M , \]
where $\bsym{a}$ is any weakly proregular sequence generating $\a$.
Lemma \ref{lem:260} says that this is a quasi-isomorphism.
\end{proof}

\begin{cor} \label{cor:15}
The subcategory $\cat{D}(\cat{Mod} A)_{\a \tup{-tor}}$ is the essential image
of the functor
\[ \mrm{R} \Gamma_{\a} : \cat{D}(\cat{Mod} A) \to \cat{D}(\cat{Mod} A) . \]
\end{cor}

\begin{proof}
Clear from Corollary \ref{cor:11}.
\end{proof}

\begin{cor} \label{cor:17}
There is equality 
\[ \cat{D}(\cat{Mod} A)_{\a \tup{-tor}} = 
\cat{D}_{\a \tup{-tor}}(\cat{Mod} A) . \]
In other words, a complex $M$ is cohomologically $\a$-torsion if and only if
all its cohomology modules $\opn{H}^i(M)$ are $\a$-torsion.
\end{cor}

This result is similar to \cite[Proposition 5.2.1(a)]{AJL2}.

\begin{proof}
One inclusion is clear -- see (\ref{eqn:54}). 
For the other direction, we have to show that if
$M \in \cat{D}_{\a \tup{-tor}}(\cat{Mod} A)$ then 
$\sigma^{\mrm{R}}_M $ is an isomorphism.
By Corollary \ref{cor:kosz.1} we can replace 
$\sigma^{\mrm{R}}_{M}$ with 
\[ e_{\bsym{a}, \infty}^{\vee} \ot 1_M :
\opn{K}^{\vee}_{\infty}(A; \bsym{a}) \ot_A M \to M , \]
where $\bsym{a}$ is any weakly proregular sequence generating $\a$.
The way-out argument of \cite[Proposition I.7.1]{RD}
says we can assume $M$ is a single $\a$-torsion module. But then
$\opn{K}^{\vee}_{\infty}(A; \bsym{a})^i \ot_A M = 0$
for all $i > 0$, so 
$e_{\bsym{a}, \infty}^{\vee} \ot 1_M$ 
is an isomorphism of complexes.
\end{proof}

Here is an example showing that a similar statement for cohomologically
complete complexes is false.

\begin{exa} \label{exa:1}
Let $A := \K[[t]]$, the power series ring in the variable $t$ over a field
$\K$, and $\a := (t)$. 
As shown in \cite[Example 3.20]{Ye2}, there is a complex 
\[ P = \bigl( \cdots \to 0 \to  P^{-1} \xar{\d}  P^{0} \to 0 \to \cdots \bigr)
\]
in which $P^{-1}$ and $P^0$ are $\a$-adically free $A$-modules
(both are $\a$-adic completions of a countable rank free $A$-module),
$\mrm{H}^{-1} (P) = 0$, and the module $\mrm{H}^0 (P)$ {\em is not
$\a$-adically complete}. 

On the other hand, the complex $P$ {\em is $\a$-adically cohomologically
complete}. To see this, we note that the $A$-modules $P^i$ are flat (see
\cite[Theorem 3.4]{Ye2}),
and hence by Proposition \ref{prop:26} the morphism 
$\xi_P : \mrm{L} \Lambda_{\a} (P) \to \Lambda_{\a} (P)$ is an isomorphism. On
the other hand the modules $P^i$ are $\a$-adically complete, so 
$\tau_P : P \to \Lambda_{\a} (P)$ is an isomorphism. Therefore 
$\tau^{\mrm{L}}_P : P \to \mrm{L} \Lambda_{\a} (P)$
is an isomorphism.
\end{exa}

\begin{thm} \label{thm:253}
If $A$ is noetherian, then every finite sequence in $A$ is weakly proregular,
so that every ideal in $A$ is weakly proregular.
\end{thm}

\begin{proof}
It is enough to prove that every finite sequence 
$\bsym{a} = (a_1, \ldots, a_n)$ is weakly proregular. In view of Theorem
\ref{thm:252}, it suffices to prove that for any injective module $I$ and any 
$k > 0$ the $A$-module 
$\mrm{H}^k (\opn{K}^{\vee}_{\infty}(A; \bsym{a}) \otimes_A I)$
is zero. 

We use the structure theory for injective modules over noetherian rings. 
Because cohomology and tensor product commute with infinite direct sums, 
it suffices to consider an indecomposable injective $A$-module; so 
assume $I$ is the injective hull of $A / \p$ for some prime ideal $\p$. This is
a $\p$-torsion module, and also an $A_{\p}$-module. 

If $\a \subset \p$ then each $a_i \in \p$, so 
$A[a_i^{-1}] \otimes_A I = 0$. This says that 
$\opn{K}^{\vee}_{\infty}(A; \bsym{a})^k \otimes_A I = 0$
for all $k > 0$. 

Next assume that $\a \not\subset \p$. Then for at least one index $i$ we have
$a_i \notin \p$, so that $a_i$ is invertible in $A_{\p}$.
This implies that the homomorphism 
\[ \opn{K}^{\vee}_{\infty}(A; a_i)^0 \otimes_A I \to 
\opn{K}^{\vee}_{\infty}(A; a_i)^1 \otimes_A I  \]
is bijective. So the complex 
$\opn{K}^{\vee}_{\infty}(A; a_i) \otimes_A I$
is acyclic. Now 
\[ \opn{K}^{\vee}_{\infty}(A; \bsym{a}) \otimes_A I \cong 
\opn{K}^{\vee}_{\infty}(A; \bsym{b}) \otimes_A  
\opn{K}^{\vee}_{\infty}(A; a_i) \otimes_A I , \]
where $\bsym{b}$ is the subsequence of $\bsym{a}$ obtained by deleting $a_i$. 
Therefore the complex 
$\opn{K}^{\vee}_{\infty}(A; \bsym{a}) \otimes_A I$ is acyclic. 
\end{proof}

Here is a pretty natural example of a weakly proregular sequence in a
non-noetherian ring. There is a follow-up in Example \ref{exa:280}.

\begin{exa} \label{exa:251}
Let $\K$ be a field, and let $A$ and $B$ be adically complete noetherian
$\K$-algebras, with defining ideals $\a$ and $\b$ respectively. 
Take $C := A \ot_{\K} B$. The ring $C$ is often {\em not
noetherian}. 

This happens for instance if $\K$ has characteristic $0$, and
$A = B := \K[[t]]$, the ring of power series in a variable $t$. 
Let $\mfrak{d} \subset C$ be the kernel of the multiplication map 
$C  = A \otimes_{\K} A \to A$. The ideal $\mfrak{d}$ is not finitely generated. 
To see why, note that $\mfrak{d} / \mfrak{d}^2 \cong \Omega^1_{A / \K}$,
and $L \ot_C \Omega^1_{A / \K} \cong \Omega^1_{L / \K}$,
where $L := \K((t))$. Since $L / \K$ is a separable field extension of infinite
transcendence degree, it follows that the rank of $\Omega^1_{L / \K}$
is infinite. 

Let's return to the general situation above. 
Choose finite generating sequences 
$\bsym{a} = (a_1, \ldots, a_m)$ and $\bsym{b} = (b_1, \ldots, b_n)$ for $\a$
and $\b$ respectively. By Theorem \ref{thm:253} these sequences are weakly
proregular. Consider the sequence 
\[ \bsym{c} := (a_1 \ot 1, \ldots, a_m \ot 1, 1 \ot b_1, \ldots, 1 \ot b_n) \]
in $C$. We claim that this sequence {\em is weakly proregular}. 
The reason is that for every $i$ there is a canonical isomorphism of 
DG algebras 
\[ \mrm{K}(A; \bsym{a}^i) \ot_{\K} \mrm{K}(B; \bsym{b}^i) \cong
\mrm{K}(C; \bsym{c}^i) . \]
By the K\"unneth formula we get isomorphisms of $C$-modules 
\[ \mrm{H}^k (\mrm{K}(C; \bsym{c}^i)) \cong
\bigoplus_{k \leq l \leq 0} 
\mrm{H}^l (\mrm{K}(A; \bsym{a}^i)) \ot_{\K} 
\mrm{H}^{k-l} (\mrm{K}(B; \bsym{b}^{i}))  \]
for every $k \leq 0$, 
compatible with $i$. Thus for every $k < 0$ the inverse system
$\bigl\{ \mrm{H}^k (\mrm{K}(C; \bsym{c}^i)) \bigr\}_{i \in \N}$
is pro-zero.
\end{exa}

\section{The Telescope Complex} 
\label{sec:tel}

The purpose of this section is to prove Theorem \ref{thm:250}. 

Let $A$ be a commutative ring (not necessarily noetherian). For a set $X$ and an
$A$-module $M$ we denote by
$\opn{F}(X, M)$ the set of all functions $f : X \to M$.
This is an $A$-module in the obvious way. We denote by 
$\opn{F}_{\mrm{fin}}(X, M)$ the submodule of $\opn{F}(X, M)$ consisting of
functions with finite support. Note that 
$\opn{F}_{\mrm{fin}}(X, A)$ is a free
$A$-module with basis the delta functions $\delta_x : X \to A$.
(This notation comes from \cite{Ye2}.)

\begin{dfn} \label{dfn:5} 
\begin{enumerate}
\item Given an element $a \in A$, the {\em telescope complex}
$\opn{Tel}(A; a)$ is the complex
\[ \opn{Tel}(A; a) :=
\bigl( \cdots \to 0 \to \opn{F}_{\mrm{fin}}( \N, A) \xar{\d}
\opn{F}_{\mrm{fin}}( \N, A) \to 0 \to \cdots \bigr) \]
concentrated in degrees $0$ and $1$. The differential $\d$ is
\[ \d(\delta_i) := 
\begin{cases}
\delta_0  & \tup{ if } i = 0 , \\
\delta_{i- 1} - a \delta_{i} & \tup{ if } i \geq 1 .
\end{cases} \]

\item Given a sequence $\bsym{a} = (a_1, \ldots, a_n)$ of elements of $A$, we
define
\[ \opn{Tel}(A; \bsym{a}) :=
\opn{Tel}(A; a_1) \otimes_A \cdots \otimes_A \opn{Tel}(A; a_n) . \]
\end{enumerate}
\end{dfn}

Note that $\opn{Tel}(A; \bsym{a})$ is a complex of free $A$-modules,
concentrated in degrees $0, \ldots, n$. 
This complex has an obvious functoriality in $(A; \bsym{a})$. 

Recall that for $j \in \N$ we write 
$[0, j] = \{ 0, \ldots, j \}$.
We view $\mrm{F}([0, j], A)$ as the free submodule of
$\opn{F}_{\mrm{fin}}(\N, A)$ with basis 
$\{ \delta_i \}_{ i  \in [0, j]}$.

Let $j \in \N$.
For any $a \in A$ let $\opn{Tel}_j(A; a)$ be the subcomplex 
\[ \opn{Tel}_j(A; a) :=
\bigl( \cdots \to 0 \to \opn{F}([0, j] , A) \xar{\d}
\opn{F}([0, j] , A) \to 0 \to \cdots \bigr) \]
of $\opn{Tel}(A; a)$.
For the sequence $\bsym{a} = (a_1, \ldots, a_n)$ we define 
\[ \opn{Tel}_j(A; \bsym{a}) :=
\opn{Tel}_j(A; a_1) \otimes_A \cdots \otimes_A \opn{Tel}_j(A; a_n) . \]
This is a subcomplex of $\opn{Tel}(A; \bsym{a})$. It is clear that 
\begin{equation} \label{eqn:214}
\opn{Tel}(A; \bsym{a}) = \bigcup_{j \geq 0} \, \opn{Tel}_j(A; \bsym{a}) . 
\end{equation}

Recall the dual Koszul complex 
$\opn{K}^{\vee}_{}(A; \bsym{a})$ from formula (\ref{eqn:5.2}). 
For any $j \geq 0$ we define a homomorphism of complexes
\begin{equation} \label{eqn:204}
w_{a, j} : \opn{Tel}_j(A; a) \to \opn{K}^{\vee}_{}(A; a^j) 
\end{equation}
as follows, using the presentation (\ref{eqn:203}) of 
$\opn{K}^{\vee}_{}(A; a^j)$. 
In degree $0$ the homomorphism 
\[ w_{a, j}^0 : \opn{Tel}_j(A; a)^0 = \opn{F}([0, j] , A) \to 
\opn{K}^{\vee}_{}(A; a^j)^0 = A \]
is defined to be
\[  w_{a, j}^0(\delta_i) :=
\begin{cases}
1  & \tup{ if } i = 0 , \\
0 & \tup{ if } i \geq 1 .
\end{cases} \]
In degree $1$ the homomorphism 
\[ w_{a, j}^1 : \opn{Tel}^j(A; a)^1 = \opn{F}([0, j] , A) \to 
\opn{K}^{\vee}_{}(A; a^j)^1 = A \]
is defined to be 
$w_{a, j}^1(\delta_i) := a^{j-i}$. 
This makes sense since $i \in [0, j]$.

For a sequence $\bsym{a} = (a_1, \ldots, a_n)$ we define 
\begin{equation}
w_{\bsym{a}, j} :=
w_{a_1, j} \otimes \cdots \otimes  w_{a_n, j} :
\opn{Tel}_j(A; \bsym{a}) \to \opn{K}^{\vee}_{}(A; \bsym{a}^j) . 
\end{equation}
The homomorphisms of complexes $w_{\bsym{a}, j}$ are functorial in $j$, so in
the direct limit we get a homomorphism of complexes 
\begin{equation} \label{eqn:206}
w_{\bsym{a}} := \lim_{j \to} w_{\bsym{a}, j} : \opn{Tel}(A; \bsym{a}) \to 
\opn{K}^{\vee}_{\infty}(A; \bsym{a}) .
\end{equation}
Of course 
$w_{\bsym{a}} =  w_{a_1} \ot \cdots \ot  w_{a_n}$.
Let us also define 
\begin{equation}
u_{\bsym{a}} : \opn{Tel}(A; \bsym{a}) \to A \ , \
u_{\bsym{a}} := e_{\bsym{a}, \infty}^{\vee} \circ w_{\bsym{a}} \ ;
\end{equation}
cf.\ (\ref{eqn:213}).

\begin{lem} \label{lem:201}
The homomorphism $w_{\bsym{a}, j}$ is a homotopy equivalence, and the
homomorphism $w_{\bsym{a}}$ is a quasi-isomorphism.
\end{lem}

\begin{proof}
First consider the case $n = 1$, $A = \Z[t]$, the polynomial ring in the 
variable $t$, and $a = t$. The fact that $w_{t,j}$ is a quasi-isomorphism is an
easy calculation, once we notice that 
\[ \mrm{H}^0 \bigl( \opn{Tel}_j(\Z[t]; t) \bigr) = 
\mrm{H}^0 \bigl( \opn{K}^{\vee}_{}(\Z[t]; t^j) \bigr) = 0 , \]
and 
\[ \mrm{H}^1 \bigl( \opn{Tel}_j(\Z[t]; t) \bigr) \cong 
\mrm{H}^1 \bigl( \opn{K}^{\vee}_{}(\Z[t]; t^j) \bigr) \cong \Z[t] / (t^j) . \]

Next, for any $(A; a)$ we have a ring homomorphism $\Z[t] \to A$ sending 
$t \mapsto a$. Since $w_{a, j}$ is gotten from 
$w_{t, j}$ by the base change $A \ot_{\Z[t]} -$, and since 
$\opn{Tel}_j(\Z[t]; t)$ and $\opn{K}^{\vee}_{}(\Z[t]; t^j)$ are bounded
complexes of flat $\Z[t]$-modules, it follows that $w_{a, j}$ is also a
quasi-isomorphism. 

The flatness argument, with induction, also proves that for sequence $\bsym{a}$
of length $n \geq 2$ the homomorphism $w_{\bsym{a}, j}$ is a quasi-isomorphism.
Because $\opn{Tel}_j(A; \bsym{a})$ and $\opn{K}^{\vee}(A; \bsym{a}^j)$ are
bounded complexes of free $A$-modules, it follows that $w_{\bsym{a}, j}$ is a
homotopy equivalence.

Finally going to the direct limit preserves exactness, so $w_{\bsym{a}}$ is
a quasi-iso\-mor\-phism.
\end{proof}

Warning: the quasi-isomorphism $w_{\bsym{a}}$ is {\em not a homotopy
equivalence} (except in trivial cases). 

\begin{prop} \label{prop:251}
Let $\bsym{a}$ be a weakly proregular sequence in $A$, and $\a$ the ideal
generated by $\bsym{a}$.
For any $M \in \cat{D}(\cat{Mod} A)$ there is an isomorphism
\[ v^{\mrm{R}}_{\bsym{a}, M} : \mrm{R} \Gamma_{\a} (M) \to 
\opn{Tel}(A; \bsym{a}) \otimes_A M \]
in $\cat{D}(\cat{Mod} A)$. The isomorphism $v^{\mrm{R}}_{\bsym{a}, M}$ is
functorial in $M$, and satisfies 
\[ (u_{\bsym{a}} \ot 1_M) \circ v^{\mrm{R}}_{\bsym{a}, M}=
\sigma^{\mrm{R}}_M  \]
as morphisms $\mrm{R} \Gamma_{\a}  (M) \to M$.
\end{prop}

\begin{proof}
Combine Lemma \ref{lem:201} and Corollary \ref{cor:kosz.1}.
\end{proof}

Let us denote by $\a$ the ideal of $A$ generated by the sequence 
$\bsym{a} = (a_1, \ldots, a_n)$.
Recall that $A_j = A / \a^{j+1}$. Since 
$\a^{j n} \subset (\bsym{a}^j) \subset \a^j$
it follows that the canonical homomorphism
\begin{equation} \label{eqn:5.1}
\lim_{\leftarrow j} \, \bigl( A / (\bsym{a}^{j + 1}) \otimes_A M \bigr) \to
\lim_{\leftarrow j} \, \bigl( A_j \otimes_A M \bigr)  = \Lambda_{\a} (M)
\end{equation}
is bijective for any module $M$.

Let us write
\begin{equation}
 \opn{Tel}_j^{\vee}(A; \bsym{a}) := 
\opn{Hom}_A \bigl( \opn{Tel}_j(A; \bsym{a}), A \bigr) .
\end{equation}
We refer it as the {\em dual telescope complex}. 
Note that $\opn{Tel}_j^{\vee}(A; \bsym{a})$ is a complex of finite rank free
$A$-modules, concentrated in degrees $-n, \ldots, 0$.
The dual of the homomorphism $w_{\bsym{a}, j}$ 
is \begin{equation}
w_{\bsym{a}, j}^{\vee} : \opn{K}_{}(A; \bsym{a}^j) 
 \to \opn{Tel}_j^{\vee}(A; \bsym{a}) .
\end{equation}
Since $w_{\bsym{a}, j}$ is a homotopy equivalence, it
follows that $w_{\bsym{a}, j}^{\vee}$ is also a homotopy equivalence. 
Therefore
\[ \mrm{H}^0 (w_{\bsym{a}, j}^{\vee}) : 
\mrm{H}^0 ( \opn{Tel}_j^{\vee}(A; \bsym{a}) ) 
\to \mrm{H}^0 ( \opn{K}_{}(A; \bsym{a}^j) ) \]
is an isomorphism of $A$-modules. Define  
\begin{equation} \label{eqn:216}
\opn{tel}_{\bsym{a}, j} : 
\opn{Tel}_j^{\vee}(A; \bsym{a}) \to A / (\bsym{a}^j)
\end{equation}
to be the unique homomorphism of complexes such that 
\[ \mrm{H}^0 (\opn{tel}_{\bsym{a}, j}) \circ 
\mrm{H}^0 (w_{\bsym{a}, j}^{\vee})^{-1} : 
\mrm{H}^0 ( \opn{K}_{}(A; \bsym{a}^j) ) \to  A / (\bsym{a}^j) \]
is the canonical $A$-algebra  isomorphism (\ref{eqn:217}). 

For any $M \in \cat{C}(\cat{Mod} A)$ and $j \in \N$ 
there is a canonical isomorphism of complexes
\begin{equation} \label{eqn:tel.12}
\opn{Hom}_A \bigl( \opn{Tel}_j(A; \bsym{a}), M \bigr) \cong 
\opn{Tel}^{\vee}_j(A; \bsym{a}) \otimes_A M . 
\end{equation}
There is also  a canonical isomorphism of complexes
\begin{equation} \label{eqn:tel.6}
\opn{Hom}_A \bigl( \opn{Tel}(A; \bsym{a}), M \bigr) \cong \lim_{\leftarrow j} \,
\opn{Hom}_A \bigl( \opn{Tel}_j(A; \bsym{a}), M \bigr) 
\end{equation}
coming from (\ref{eqn:214}).
We define a homomorphism of complexes 
\begin{equation}
\begin{aligned}
\opn{tel}_{\bsym{a}, M, j} & : 
\opn{Hom}_A \bigl( \opn{Tel}_j(A; \bsym{a}), M \bigr) \to 
A / (\bsym{a}^j) \otimes_A M \ ,
\\
\opn{tel}_{\bsym{a}, M, j} & := \opn{tel}_{\bsym{a}, j} \otimes \, 1_M \ ,
\end{aligned}
\end{equation}
using the isomorphism (\ref{eqn:tel.12}).

\begin{dfn} \label{dfn:280}
For any $M \in \cat{C}(\cat{Mod} A)$ let 
\[ \opn{tel}_{\bsym{a}, M} : 
\opn{Hom}_A \bigl( \opn{Tel}(A; \bsym{a}), M \bigr) \to  \Lambda_{\a} (M) \]
be the homomorphism of complexes 
\[ \opn{tel}_{\bsym{a}, M} := 
\lim_{\leftarrow j} \, \opn{tel}_{\bsym{a}, M, j} 
= \lim_{\leftarrow j} \, (\opn{tel}_{\bsym{a}, j} \otimes \, 1_M) \ . \]
Here we use the isomorphisms (\ref{eqn:tel.6}) and (\ref{eqn:5.1}).
\end{dfn}

Note that $\opn{tel}_{\bsym{a}, M}$ is functorial in $M$.

\begin{rem} \label{rem:tel.1}
For a module $M$ the homomorphism 
\[ \opn{tel}_{\bsym{a}, M} :  
\opn{Hom}_A \bigl( \opn{Tel}(A; \bsym{a})^0, M \bigr) \to \La_{\a}(M) \]
can be expressed explicitly as an $\a$-adically convergent power series. First
we note that an element 
$f \in \opn{Hom}_A \bigl( \opn{Tel}(A; \bsym{a})^0, M \bigr)$
is the same as a function $f : \N^n \to M$.
For $a \in A$ and $i \in \N$ we define the ``modified $i$-th power'' 
$p(a, i) \in A$ to be
$p(a, 0) := 1$, $p(a, 1) := -1$ and 
$p(a, i) := -a^{i-1}$ if $i \geq 2$.
Then 
\begin{equation} \label{eqn:tel.13}
 \opn{tel}_{\bsym{a}, M}(f) = \sum_{(i_1, \ldots, i_n) \in \N^n} \ 
p(a_1, i_1) \cdots p(a_n, i_n)  f(i_1, \ldots, i_n) \in 
\Lambda_{\a} (M) .
\end{equation}
We shall not require this formula.
\end{rem}

Consider the homomorphism of complexes 
\begin{equation}
\opn{Hom}(u_{\bsym{a}}, 1_M) : M \cong \opn{Hom}_A(A, M) \to 
\opn{Hom}_A \bigl( \opn{Tel}(A; \bsym{a}), M \bigr) 
\end{equation}
induced by 
$u_{\bsym{a}} : \opn{Tel}(A; \bsym{a}) \to A$. 

\begin{lem} \label{lem:203}
For any $M \in \cat{Mod} A$ there is equality
$\opn{tel}_{\bsym{a}, M} \circ \opn{Hom}(u_{\bsym{a}}, 1_M) = \tau_M$ , 
as homomorphisms 
$M \to \Lambda_{\a}(M)$.
\end{lem}

\begin{proof}
It suffices to prove that for every $j \geq 0$ there is equality 
\[ \opn{tel}_{\bsym{a}, M,j} \circ \opn{Hom}(u_{\bsym{a}, j}, 1_M)  = 
f_j \circ 1_M \]
as homomorphisms $M \to A / (\bsym{a}^j)$, where 
$f_j : A \to  A / (\bsym{a}^j)$ is the canonical ring homomorphism, and
$u_{\bsym{a}, j} := e_{\bsym{a}, j}^{\vee} \circ w_{\bsym{a} ,j}$.
But everything is functorial in $M$, so we can restrict attention to $M = A$. 
Thus we have to show that 
$\opn{tel}_{\bsym{a}, j} \circ\, u_{\bsym{a}, j}^{\vee} = f_j$.

Consider the diagram
\[ \UseTips \xymatrix @C=7ex @R=7ex {
A
\ar[r]^{e_{\bsym{a}, j}}
\ar[dr]_{u_{\bsym{a}, j}^{\vee}}
\ar@(u,u)[rr]^{f_j}
&
\opn{K}(A; \bsym{a}^j)
\ar[r]^{g_j}
\ar[d]^{w_{\bsym{a} ,j}}
&
A / (\bsym{a}^j)
\\
& 
\opn{Tel}_j^{\vee}(A; \bsym{a})
\ar[ur]_{\opn{tel}_{\bsym{a}, j} }
} \]
where $g_j$ is the DG algebra homomorphism. 
By definition the three triangles are commutative. Hence the whole diagram is
commutative.
\end{proof}

\begin{thm} \label{thm:250}
Let $A$ be any ring, let $\bsym{a}$ be a weakly proregular sequence in $A$, 
and let $P$ be a flat  $A$-module. Then  the homomorphism
\[ \opn{tel}_{\bsym{a}, P} : 
\opn{Hom}_A \bigl( \opn{Tel}(A; \bsym{a}), P \bigr) \to \Lambda_{\a} (P) \]
is a quasi-isomorphism.
\end{thm}

\begin{proof}
Given an inverse system $\{ M_j \}_{j \in \N}$ of complexes of abelian
groups, for every integer $k$ there is a canonical homomorphism
\[ \psi^k : \mrm{H}^k \Bigl( \lim_{\leftarrow j} \, M_j \Bigr) \to 
\lim_{\leftarrow j} \, \bigl( \mrm{H}^k (M_j) \bigr) . \]

By definition of $\opn{tel}_{\bsym{a}, P}$, for $k = 0$ there is a
commutative diagram
\[ \UseTips \xymatrix @C=18ex @R=6ex {
\mrm{H}^k \Bigl(  \opn{Hom}_A \bigl( \opn{Tel}(A; \bsym{a}), P \bigr) \Bigr)
\ar[r]^{\mrm{H}^k (\opn{tel}_{\bsym{a}, P}) }
\ar[d]_{\cong}
&
\Lambda_{\a} (P)
\ar[d]^{\cong}
\\
\mrm{H}^k \Bigl( \lim_{\leftarrow j} \,  
\bigl( \opn{Tel}^{\vee}_j(A; \bsym{a}) \ot_A P \bigr) \Bigr) 
\ar[r]^(0.54){\mrm{H}^k ( \lim_{\leftarrow j} \, \opn{tel}_{\bsym{a}, P, j}) }
\ar[d]_{\psi^k}
&
\lim_{\leftarrow j} \,  \bigl( ( A / (\bsym{a}^j) ) \ot_A P \bigr)
\\
\lim_{\leftarrow j} \, \mrm{H}^k 
\bigl( \opn{Tel}^{\vee}_j(A; \bsym{a}) \ot_A P \bigr) 
\ar[ur]_(0.54){\quad \quad \lim_{\leftarrow j} \, \mrm{H}^k ( 
\opn{tel}_{\bsym{a}, P, j}) }
} \]
The left part of the diagram makes sense for every $k$.
We will prove that:
\begin{enumerate}
\item 
$\lim_{\leftarrow j} \, \mrm{H}^k 
\bigl( \opn{Tel}^{\vee}_j(A; \bsym{a}) \ot_A P \bigr) = 0$
for all $k \neq 0$. 

\item 
$\mrm{H}^0 ( \opn{tel}_{\bsym{a}, P, j})$ is bijective for every $j \geq 0$.

\item $\psi^k$ is bijective for every $k$.
\end{enumerate}
Together these imply that 
$\mrm{H}^k ( \opn{tel}_{\bsym{a}, P})$ is bijective for every $k$.

There are quasi-isomorphisms 
\[ w_{\bsym{a}, j}^{\vee}  : \opn{K}(A; \bsym{a}^j)  \to
\opn{Tel}^{\vee}_j(A; \bsym{a})  \]
that are compatible with $j$.
Since $P$ is flat, according to Corollary \ref{cor:293} we get induced
isomorphisms 
\begin{equation} \label{eqn:254}
\mrm{H}^k \bigl( \opn{Tel}^{\vee}_j(A; \bsym{a}) \ot_A P \bigr) \cong
\mrm{H}^k \bigl( \opn{K}(A; \bsym{a}^j) \ot_A P \bigr) \cong
\mrm{H}^k \bigl( \opn{K}(A; \bsym{a}^j) \bigr)  \ot_A P
\end{equation}
that are compatible with $j$.

There is a canonical ring isomorphism 
$\mrm{H}^0 \bigl( \opn{K}(A; \bsym{a}^j) \bigl) \cong A / (\bsym{a}^j)$.
By definition of $\opn{tel}_{\bsym{a}, j}$, the homomorphism 
\[ \mrm{H}^0 ( \opn{tel}_{\bsym{a}, j}) : 
\mrm{H}^0 \bigl( \opn{Tel}^{\vee}_j(A; \bsym{a}) \bigr) \to
A / (\bsym{a}^j) \]
is bijective. 
Hence, using Corollary \ref{cor:293} again, we see that
$\mrm{H}^0 ( \opn{tel}_{\bsym{a}, P, j})$ is also
bijective. This proves (2).

We are given that $\bsym{a}$ is a weakly proregular sequence, which means that
the homomorphism
\[  \mrm{H}^k(p_{\bsym{a}, j', j}) :
 \mrm{H}^k \bigl( \opn{K}(A; \bsym{a}^{j'}) \bigr)   \to 
\mrm{H}^k \bigl( \opn{K}(A; \bsym{a}^j) \bigr)  \]
is zero for $k < 0$ and $j' \gg j$.
As for $k = 0$, we know that  
\[ \mrm{H}^0 \bigl( \opn{K}(A; \bsym{a}^{j'}) \bigr)   \to 
\mrm{H}^0 \bigl( \opn{K}(A; \bsym{a}^j) \bigr)  \]
is surjective for $j' \geq j$. Of course 
$\mrm{H}^k \bigl( \opn{K}(A; \bsym{a}^j) \bigr) = 0$ 
for $k > 0$. 
Thus for every $k$ the inverse system of modules 
\[ \bigl\{  \mrm{H}^k \bigl( \opn{Tel}^{\vee}_j(A; \bsym{a}) \ot_A P \bigr) 
\bigr\}_{j \in \N}  \cong
\bigl\{  \mrm{H}^k \bigl( \opn{K}(A; \bsym{a}^j) \bigr) \ot_A P 
\bigr\}_{j \in \N}  \]
satisfies the Mittag-Leffler condition.

The inverse system of complexes 
$\bigl\{ \opn{Tel}^{\vee}_j(A; \bsym{a}) \ot_A P \bigr\}_{j \in \N}$
also satisfies the Mittag-Leffler  condition, since it has surjective transition
maps.  (Warning: see Remark \ref{rem:251}.)
Therefore, by  \cite[Proposition 1.1.24]{KS1} or \cite[Theorem 3.5.8]{We},
the homomorphisms
\[ \psi^k : 
\mrm{H}^k \Bigl( 
\lim_{\leftarrow j} \, \bigl( \opn{Tel}^{\vee}_j(A; \bsym{a}) \ot_A P \bigr) 
\Bigr) \to
\lim_{\leftarrow j} \, \mrm{H}^k
\bigl( \opn{Tel}^{\vee}_j(A; \bsym{a}) \ot_A P \bigr) 
\]
are bijective. Thus (3) is true. 

Finally, weak proregularity, with the isomorphisms (\ref{eqn:254}),
tell us that the homomorphism 
\[  \mrm{H}^k \bigl( \opn{Tel}^{\vee}_{j'}(A; \bsym{a}) \ot_A P \bigr)   \to 
\mrm{H}^k \bigl( \opn{Tel}^{\vee}_j(A; \bsym{a}) \ot_A P \bigr)  \]
is zero for $k < 0$ and $j' \gg j$.
And everything is zero for $k > 0$. This implies (1). 
\end{proof}

\begin{cor} \label{cor:294}
Assume $\bsym{a}$ is a weakly proregular sequence in $A$. Then
for every K-flat complex $P$ the homomorphism
\[ \opn{tel}_{\bsym{a}, P} : \opn{Hom}_A ( \opn{Tel}(A; \bsym{a}), P) \to 
\Lambda_{\a}(P)  \]
is a quasi-isomorphism.
\end{cor}

\begin{proof}
By Proposition \ref{prop:28} we can assume that $P$ is a complex of flat
modules. By Proposition \ref{prop:1.1} we reduce to
the case of a single flat module $P$. This is the theorem above. 
\end{proof}

\begin{rem} \label{rem:251}
The inverse systems of complexes 
$\bigl\{ \opn{K}(A; \bsym{a}^j) \ot_A P \bigr\}_{j \in \N}$
does not satisfy the ML condition; so we can't expect to get a
quasi-isomorphism in the inverse limit: the homomorphism
\[ \lim_{\leftarrow j} \, (w_{\bsym{a}, j}^{\vee} \ot 1_P) : 
\lim_{\leftarrow j} \, \bigl( \opn{K}(A; \bsym{a}^j) \ot_A P \bigr) \to
\lim_{\leftarrow j} \, \bigl( \opn{Tel}^{\vee}_j(A; \bsym{a}) \ot_A P \bigr) \]
will usually not be a quasi-isomorphism.

Indeed, this will even fail for the ring $A := \K[t]$, the  polynomial algebra
over a field $\K$, with sequence $\bsym{a} := (t)$ and flat module $P := A$. 
Here we get
\[ \mrm{H}^0 \bigl( \lim_{\leftarrow j} \, \opn{Tel}^{\vee}_j(A; \bsym{a}) 
\bigr) \cong  
\mrm{H}^0 \bigl( 
\opn{Hom}_A \bigl( \opn{Tel}(A; \bsym{a}), A \bigr) \bigr)
\cong \Lambda_{\a} (A) \cong \K[[t]] . \]
But 
$\lim_{\leftarrow j} \, \opn{K}(A; \bsym{a}^j)^0 \cong A$
and
$\lim_{\leftarrow j} \, \opn{K}(A; \bsym{a}^j)^{-1} = 0$,
giving 
\[ \mrm{H}^0 \bigl( \lim_{\leftarrow j} \, \opn{K}(A; \bsym{a}^j) \bigr)
\cong  A = \K[t] .  \]
\end{rem}

\begin{cor} \label{cor:tel.1}
Assume $\bsym{a}$ is a weakly proregular sequence in $A$.
For any $M \in \cat{D}(\cat{Mod} A)$ there is an isomorphism
\[ \opn{tel}_{\bsym{a}, M}^{\mrm{L}} : 
\opn{Hom}_A \bigl( \opn{Tel}(A; \bsym{a}) , M \bigr) 
\iso \mrm{L} \Lambda_{\a} (M) \]
in $\cat{D}(\cat{Mod} A)$, functorial in $M$, such that
\[ \opn{tel}_{\bsym{a}, M} \circ \opn{Hom}(u_{\bsym{a}}, 1_M) 
= \tau^{\mrm{L}}_M , \] 
as morphisms $M \to \mrm{L} \Lambda_{\a}(M)$.
\end{cor}

\begin{proof}
It is enough to consider a K-flat complex $M = P$. 
For this we combine Theorem \ref{thm:250}, Proposition \ref{prop:26} and Lemma
\ref{lem:203}.
\end{proof}

The corollary says that the diagram 
\begin{equation}
\UseTips \xymatrix @C=10ex @R=6ex {
M
\ar[d]_{\opn{Hom}(u_{\bsym{a}}, 1_M)}
\ar[dr]^{\tau^{\mrm{L}} _M}
& 
\\
\opn{Hom}_A \bigl( \opn{Tel}(A; \bsym{a}), M \bigr)
\ar[r]_(0.65){\opn{tel}^{\mrm{L}} _{\bsym{a}, M}}
&
\mrm{L} \Lambda_{\a} (M)
} 
\end{equation}
is commutative.

\begin{cor} \label{cor:261}
Let $\a$ be a weakly proregular ideal in $A$. 
The cohomological dimension of the functor $\mrm{L} \Lambda_{\a}$
is finite. Indeed, if $\a$ can be generated by a weakly proregular sequence of
length $n$, then the cohomological dimension of $\mrm{L} \Lambda_{\a}$ is at
most $n$.
\end{cor}

\begin{proof}
This is immediate from Corollary \ref{cor:tel.1}.
\end{proof}

\begin{rem} \label{rem:280}
The name ``telescope complex'' is inspired by a standard construction in
algebraic topology; see \cite{GM}. However here we are looking at a specific
complex of $A$-modules, and we prove that it has the expected homological
properties.

The result  \cite[Theorem 4.5]{Sc}, which corresponds to our Theorem
\ref{thm:250}, only talks about {\em bounded
complexes} $M$, and there is an extra assumption that {\em each $a_i$ has
bounded torsion}. Moreover, Schenzel states that the question for unbounded
complexes is {\em open as far as he knows}. We answer this in the affirmative in
our Theorem \ref{thm:250}: our result holds for unbounded complexes, and there 
is no further assumption beyond the weak proregularity of the sequence 
$\bsym{a}$. 

In \cite{AJL1} there is an assertion similar to Theorem \ref{thm:250}
(more precisely, it corresponds to Theorem \ref{thm:mgm.3}). This is
\cite[formula $\tup{(0.3)}_{\tup{aff}}$]{AJL1},
that also refers to unbounded complexes, and makes no assumption except 
proregularity of the sequence $\bsym{a}$. In \cite[Correction]{AJL1} there is
some elaboration on the specific conditions needed for the proofs to be
correct. As far as we understand, the correct conditions are weak
proregularity for $\bsym{a}$, plus bounded torsion for each $a_i$. 
Hence our Theorem \ref{thm:250}, and also our Theorem
\ref{thm:mgm.3}, appear to be stronger than the affine versions of the results
in \cite{AJL1}. 

Our proof of Theorem \ref{thm:250}  does not depend on any of the results in
either \cite{AJL1} or \cite{Sc}. We believe our proof is quite transparent. 
Note also that we give an explicit formula for the homomorphism of complexes
$\opn{tel}_{\bsym{a}, P}$, that is not found in prior papers. 
\end{rem}

\section{Permanence of Weak Proregularity} \label{sec:permanence}

In this section we show that derived completion and derived torsion are
independent of the ring, when the the ideal is weakly proregular (Theorem 
\ref{thm:6.2}). 

\begin{thm} \label{thm:290}
Let $A$ be a ring, let $\bsym{a}$ and $\bsym{b}$ be finite sequences of elements
of $A$, and let $\a := (\bsym{a})$ and $\b := (\bsym{b})$, the ideals generated
by these sequences. Assume that $\sqrt{\a} = \sqrt{\b}$. 
Then the complexes $\opn{Tel}(A; \bsym{a})$ and $\opn{Tel}(A; \bsym{b})$
are homotopy equivalent.
\end{thm}

Note that these sequences are not assumed to be weakly proregular.

\begin{proof}
For a sufficiently large positive integer $p$ we have $b_i^p \in \a$ and 
$a_j^p \in \b$ for all $i, j$. Hence there are finitely many $c_{i,j}, d_{j,i}
\in A$ such that 
$b_i^p = \sum_j c_{i,j} a_j$ and $a_j^p = \sum_i d_{j,i} b_i$.
Define
$\til{A}$ to be the quotient of the polynomial ring 
$\Z[ \{ s_i, t_j, u_{i, j}, v_{j, i} \} ]$ in finitely many variables, modulo
the relations 
$t_i^p = \sum_j u_{i,j} s_j$ and $s_j^p = \sum_i v_{j,i} t_i$.
Let $\til{a}_i \in \til{A}$ and $\til{b}_j \in \til{A}$ be the images of $s_i$
and $t_j$ respectively. 
There is a ring homomorphism $f : \til{A} \to A$ such that 
$f(\til{a}_i) = a_i$ and $f(\til{b_j}) = b_j$.

Define the finite sequences $\til{\bsym{a}} := (\til{a}_1, \ldots)$ and 
$\til{\bsym{b}} := (\til{b}_1, \ldots)$. There are corresponding ideals 
$\til{\a} := (\til{\bsym{a}})$ and $\til{\b} := (\til{\bsym{b}})$ in
$\til{A}$. Since the ring $\til{A}$ is noetherian, the sequences 
$\til{\bsym{a}}$ and  $\til{\bsym{b}}$ are weakly proregular.  
By construction we have $\sqrt{\til{\a}} = \sqrt{\til{\b}}$, and therefore 
$\Ga_{\til{\a}} = \Ga_{\til{\b}}$ as functors. According to Proposition
\ref{prop:251} there are isomorphisms 
\[ \opn{Tel}(\til{A}; \til{\bsym{a}}) \cong \mrm{R} \Ga_{\til{\a}} (\til{A})
\cong \mrm{R} \Ga_{\til{\b}} (\til{A}) \cong 
\opn{Tel}(\til{A}; \til{\bsym{b}}) \]
in $\cat{D}(\cat{Mod} \til{A})$.
Now $\opn{Tel}(\til{A}; \til{\bsym{a}})$ and 
$\opn{Tel}(\til{A}; \til{\bsym{b}})$
are bounded complexes of free $\til{A}$-modules, so there is a homotopy
equivalence 
$\til{\phi} : \opn{Tel}(\til{A}; \til{\bsym{a}}) \to 
\opn{Tel}(\til{A}; \til{\bsym{b}})$.
Applying base change along  $f$ to $\til{\phi}$ we get a homotopy
equivalence 
$\phi : \opn{Tel}(A; \bsym{a}) \to \opn{Tel}(A; \bsym{b})$
over $A$.
\end{proof}

The next result was already proved by Schenzel \cite[Lemma 3.3]{Sc}.

\begin{cor} \label{cor:300}  
In the situation of Theorem \tup{\ref{thm:290}}, the sequence  $\bsym{a}$ is
weakly proregular if and only if the sequence $\bsym{b}$ is weakly proregular.
\end{cor}

\begin{proof}
By Lemma \ref{lem:201} there are quasi-isomorphisms 
$w_{\bsym{a}}  : \opn{Tel}(A; \bsym{a}) \to 
\opn{K}^{\vee}_{\infty}(A; \bsym{a})$
and 
$w_{\bsym{b}}  : \opn{Tel}(A; \bsym{b}) \to 
\opn{K}^{\vee}_{\infty}(A; \bsym{b})$.
Now all these complexes are K-flat; therefore for any $A$-module $I$
there is a diagram of quasi-isomorphisms
\[
\begin{aligned}
& \opn{K}^{\vee}_{\infty}(A; \bsym{a}) \ot_A I 
\xleftarrow{w_{\bsym{a}} \ot 1_I}
\opn{Tel}(A; \bsym{a}) \ot_A I
\\
& \qquad \xar{\phi \ot 1_I}
\opn{Tel}(A; \bsym{b}) \ot_A I
\xar{w_{\bsym{b}} \ot 1_I}
\opn{K}^{\vee}_{\infty}(A; \bsym{b}) \ot_A I .
\end{aligned} \]
Here $\phi : \opn{Tel}(A; \bsym{a}) \to \opn{Tel}(A; \bsym{b})$
is a homotopy equivalence, that exists by Theorem \ref{thm:290}.
Taking $I$ to be an arbitrary injective $A$-module, Theorem \ref{thm:252} says
that $\bsym{a}$ is weakly proregular if and only if 
$\bsym{b}$ is weakly proregular.
\end{proof}

\begin{cor} \label{cor:290}
Let $\a$ be a weakly proregular ideal in a ring $A$. Then any finite sequence
that generates $\a$ is weakly proregular.
\end{cor}

\begin{proof}
Let $\bsym{a}$ be any finite sequence that generates $\a$. Since $\a$ is weakly
proregular, it has some weakly proregular generating sequence $\bsym{b}$. 
By Corollary \ref{cor:300}, $\bsym{a}$ is also weakly proregular.
\end{proof}

\begin{cor} \label{cor:291}
Let $\a$ and $\b$ be finitely generated ideals in a ring $A$, such that 
$\sqrt{\a} = \sqrt{\b}$. 
Then $\a$ is weakly proregular if and only if $\b$ is weakly
proregular.
\end{cor}

\begin{proof}
Say $\a$ is weakly proregular. Choose a weakly proregular generating sequence
$\bsym{a}$ for $\a$. Let $\bsym{b}$ be any finite sequence that generates
$\b$. By Corollary \ref{cor:300}, $\bsym{b}$ is weakly proregular. Therefore the
ideal $\b$ is weakly proregular. 
\end{proof}

Let $f :A \to B$ be a ring homomorphism. There is a forgetful functor 
(restriction of scalars)
$F : \cat{Mod} B \to \cat{Mod} A$. 
Suppose $\a \subset A$ and $\b \subset B$ are finitely generated ideals such
that 
$\sqrt{\b} = \sqrt{B \cdot f(\a)}$ in $B$.
It is easy to see that there are isomorphisms 
$F \circ \Gamma_{\b} \cong \Gamma_{\a} \circ F$
and 
$F \circ \Lambda_{\b} \cong \Lambda_{\a} \circ F$,
as functors $\cat{Mod} B \to \cat{Mod} A$. 

Sometimes such isomorphisms exist also for the derived functors. Note that the
forgetful functor $F$ is exact, so it extends to a triangulated functor 
$F : \cat{D}(\cat{Mod} B) \to \cat{D}(\cat{Mod} A)$. 

\begin{thm} \label{thm:6.2}
Let $f :A \to B$ be a homomorphism of rings, 
let $\a$ be an ideal in $A$, and let $\b$ be an ideal in $B$. 
Assume that the ideals $\a$ and $\b$ are weakly proregular, and that
$\sqrt{\b} = \sqrt{B \cdot f(\a)}$.
Then there are isomorphisms
\[ F \circ \mrm{R} \Gamma_{\b} \cong \mrm{R} \Gamma_{\a} \circ F \]
and 
\[ F \circ \mrm{L} \Lambda_{\b} \cong \mrm{L} \Lambda_{\a} \circ F \]
of triangulated functors 
$\cat{D}(\cat{Mod} B) \to \cat{D}(\cat{Mod} A)$. 
\end{thm}

\begin{proof}
In view of Corollary \ref{cor:291} we can assume that 
$\b = B \cdot f(\a)$. 
Choose a sequence $\bsym{a} = (a_1, \ldots, a_n)$ that generates $\a$,
and let 
$\bsym{b} := (f(a_1), \ldots, f(a_n))$. According to Corollary \ref{cor:290} 
the sequences $\bsym{a}$ and $\bsym{b}$ are weakly proregular, in $A$ and $B$
respectively. 

We know that 
$\opn{Tel}(B; \bsym{b}) \cong B \otimes_A \opn{Tel}(A; \bsym{a})$
as complexes of $B$-modules. 
Take any $N \in \cat{D}(\cat{Mod} B)$. 
Using Corollary \ref{cor:tel.1} and Hom-tensor adjunction we get isomorphisms 
\[ (F \circ \, \mrm{L} \Lambda_{\b}) \, (N) \cong 
\opn{Hom}_B \bigl( \opn{Tel}(B; \bsym{b}) , N \bigr) \cong
\opn{Hom}_A \bigl( \opn{Tel}(A; \bsym{a}) , N \bigr) \cong 
(\mrm{L} \Lambda_{\a} \circ F) (N) . \] 
Likewise, using Proposition \ref{prop:251}, there are isomorphisms
\[ (F \circ \mrm{R} \Gamma_{\b}) \, (N) \cong 
\opn{Tel}(B; \bsym{b}) \otimes_B N \cong 
\opn{Tel}(A; \bsym{a}) \otimes_A N \cong 
(\mrm{R} \Gamma_{\a} \circ F) \, (N) . \]
\end{proof}

\begin{exa} \label{exa:280}
This is a continuation of Example \ref{exa:251}.
Let us assume that the ring homomorphisms $\K \to A$ and $\K \to B$ are
of formally finite type, in the sense of \cite{Ye1}.
(In the terminology of \cite{AJL2} these are pseudo finite type
homomorphisms.) Let $\c$ be the ideal in $C$ generated
by the sequence $\bsym{c}$, and define $\what{C} := \La_{\c}(C)$. 
According to \cite[Corollary 1.23]{Ye1} the ring $\what{C}$ is noetherian, and 
the homomorphism $\K \to \what{C}$ is of formally finite type. 
(E.g.\ if $A = \K[[s]]$ and $B = \K[[t]]$, with defining ideals
$\a := (s)$ and $\b := (t)$, then $\what{C} \cong \K[[s, t]]$.)
Let us denote by $\hat{\bsym{c}}$ the image of the sequence 
$\bsym{c}$ in the ring $\what{C}$, and by $\hat{\c}$ the ideal it generates.
By Theorem \ref{thm:253} the sequence 
$\hat{\bsym{c}}$ is weakly proregular. Theorem \ref{thm:6.2} says that
there are isomorphisms 
$\mrm{R} \Ga_{\c} \cong \mrm{R} \Ga_{\hat{\c}}$
and 
$\mrm{L} \La_{\c} \cong \mrm{L} \La_{\hat{\c}}$
between the derived functors. 
\end{exa}

\section{MGM Equivalence} \label{sec:MGM}

In this section $A$ is a commutative ring. We do not assume
that $A$ is noetherian or complete.
Weak proregularity was defined in Definition \ref{dfn:250}.
Recall that any finite sequence in a noetherian ring is weakly proregular, and
any ideal in a noetherian ring is weakly proregular (Theorem \ref{thm:253}).

\begin{lem} \label{lem:4.1}
Let $\bsym{a}$ be a finite sequence in $A$, let $\a$ be the ideal generated 
by $\bsym{a}$, and let $M$ be
an $A$-module. Then the homomorphism 
\[ \Lambda_{\a}(e_{\bsym{a}, \infty}^{\vee} \ot 1_M) :
\Lambda_{\a} \bigl( \opn{K}^{\vee}_{\infty}(A; \bsym{a}) \otimes_A M \bigr) 
\to \Lambda_{\a} (M) \]
\tup{(}see \tup{(\ref{eqn:213}))} is an isomorphism of complexes.
\end{lem}

\begin{proof}
Since $\opn{K}^{\vee}_{\infty}(A; \bsym{a})^0 = A$, we have 
$\opn{K}^{\vee}_{\infty}(A; \bsym{a})^0 \otimes_A M \cong M$.
We will prove that 
$\Lambda_{\a} ( \opn{K}^{\vee}_{\infty}(A; \bsym{a})^i \otimes_A M ) = 0$
for $i > 0$. Now 
$\opn{K}^{\vee}_{\infty}(A; \bsym{a})^i$ is a direct sum of modules 
$N_{i, j}$, where $N_{i, j}$ is an $A[a_j^{-1}]$-module. 
Since 
\[ (A / \a^k) \otimes_A  N_{i, j} \otimes_A M = 0 \]
for any $k \in \N$, in the limit we get 
$\Lambda_{\a}(N_{i, j} \otimes_A M) = 0$.
\end{proof}

\begin{lem} \label{lem:263}  
Let $\a$ be a weakly proregular ideal in $A$.    
For any complex $M \in \cat{D}(\cat{Mod} A)$ the morphism 
\[ \mrm{L} \Lambda_{\a}(\sigma^{\mrm{R}}_M) : 
\mrm{L} \Lambda_{\a} (\mrm{R} \Gamma_{\a} (M)) \to \mrm{L} \Lambda_{\a} (M) \]
is an isomorphism.
\end{lem}

\begin{proof}
Choose a weakly proregular generating sequence $\bsym{a}$ for the ideal $\a$,
and a K-flat resolution $P \to M$ in $\cat{C}(\cat{Mod} A)$. The complex 
$\opn{K}^{\vee}_{\infty}(A; \bsym{a}) \otimes_A P$ is also K-flat. 
By  Corollary \ref{cor:kosz.1} and Proposition \ref{prop:26}, the morphism
$\mrm{L} \Lambda_{\a}(\sigma^{\mrm{R}}_M)$
can be replaced by the homomorphism of complexes
\begin{equation} \label{eqn:34}
\Lambda_{\a} (e_{\bsym{a}, \infty}^{\vee} \ot 1_P) : 
\Lambda_{\a} \bigl( \opn{K}^{\vee}_{\infty}(A; \bsym{a}) \otimes_A P \bigr) 
\to \Lambda_{\a} (P) .
\end{equation}
But by the previous lemma, the homomorphism (\ref{eqn:34}) is actually an
isomorphism in $\cat{C}(\cat{Mod} A)$.
\end{proof}

\begin{lem} \label{lem:202}
Let $\bsym{b} = (b_1, \ldots, b_n)$ be a sequence of nilpotent elements in a
ring $B$. Then 
$u_{\bsym{b}} : \opn{Tel}(B; \bsym{b}) \to B$
is a homotopy equivalence.
\end{lem}

\begin{proof}
Recall that 
$u_{\bsym{b}} = e_{\bsym{b}, \infty}^{\vee} \circ w_{\bsym{b}}$,
where 
$ w_{\bsym{b}} : \opn{Tel}(B; \bsym{b}) \to 
\opn{K}^{\vee}_{\infty}(B; \bsym{b})$
is a quasi-isomorphism.
By formulas (\ref{eqn:202}) and (\ref{eqn:210}) we see that 
$\opn{K}^{\vee}_{\infty}(B; \bsym{b})^i = 0$ for $i > 0$,
so $e_{\bsym{b}, \infty}^{\vee}$ is an isomorphism.
We conclude that $u_{\bsym{b}}  : \opn{Tel}(B; \bsym{b}) \to B$ is a
quasi-isomorphism. 
But these are bounded complexes of free $B$-modules, and hence 
$u_{\bsym{b}}$ is a homotopy equivalence.
\end{proof}

\begin{lem} \label{lem:tel.2}
Let $\bsym{a}$ be a finite sequence in $A$, and let 
$B := A / (\bsym{a}^j)$ for some $j \geq 1$.
Let $N$ be a complex of $A$-modules, whose cohomology $\mrm{H}(N)$ is
bounded, and such that each $\mrm{H}^k(N)$ is a $B$-module. Then  the
homomorphism
\[ \opn{Hom}(u_{\bsym{a}}, 1_N) : N \to 
\opn{Hom}_A \bigl( \opn{Tel}(A; \bsym{a}), N \bigr) \]
is a quasi-isomorphism.
\end{lem}

\begin{proof}
Using smart truncation and induction on $\opn{amp}(\mrm{H}(N))$, 
i.e.\ by the way-out argument of \cite[Proposition I.7.1]{RD}, we may assume
that $N$ is a single $B$-module.

Let $\bsym{b}$ denote the image of the sequence $\bsym{a}$ in $B$. Then
$\opn{Tel}(B; \bsym{b}) \cong B \otimes_A \opn{Tel}(A; \bsym{a})$
as complexes. By Hom-tensor adjunction there is an isomorphism of complexes 
\[ \opn{Hom}_A \bigl( \opn{Tel}(A; \bsym{a}), N \bigr) \cong 
\opn{Hom}_{B} \bigl( \opn{Tel}(B; \bsym{b}), N \bigr) . \]
It suffices then to prove that 
\[  \opn{Hom}(u_{\bsym{b}}, 1_N) : N \cong \opn{Hom}_B(B, N)  \to 
\opn{Hom}_B \bigl( \opn{Tel}(B; \bsym{b}), N \bigr)  \]
is a quasi-isomorphism. By Lemma  \ref{lem:202} we know that 
$u_{\bsym{b}}$ is a homotopy equivalence; and therefore 
$\opn{Hom}(u_{\bsym{b}}, 1_N)$ is a quasi-isomorphism. 
\end{proof}

\begin{lem}  \label{lem:264}  
Let $\a$ be a weakly proregular ideal in $A$. 
For any complex $M \in \cat{D}(\cat{Mod} A)$ the morphism 
\[ \mrm{R} \Gamma_{\a}(\tau^{\mrm{L}}_M) : 
\mrm{R} \Gamma_{\a} (M) \to \mrm{R} \Gamma_{\a} (\mrm{L} \Lambda_{\a} (M)) \]
is an isomorphism.
\end{lem}

\begin{proof}
By Corollary \ref{cor:tel.1} we can replace $\tau^{\mrm{L}}_M$ with 
\[ \opn{Hom}(u_{\bsym{a}}, 1_M) : M \to 
\opn{Hom}_A \bigl( \opn{Tel}(A; \bsym{a}), M \bigr) \]
And by Proposition \ref{prop:251} we can replace 
$\mrm{R} \Gamma_{\a}(\tau^{\mrm{L}}_M)$ with 
\begin{equation} \label{eqn:257}
\begin{aligned}
& 1_{\opn{Tel}} \ot \opn{Hom}(u_{\bsym{a}}, 1_M) :
\opn{Tel}(A; \bsym{a}) \ot_A M
\\ & \quad  \qquad 
\to \opn{Tel}(A; \bsym{a}) \ot_A 
\opn{Hom}_A \bigl( \opn{Tel}(A; \bsym{a}), M \bigr) .
\end{aligned} 
\end{equation}
We will prove that (\ref{eqn:257}) is a quasi-isomorphism.

In view of Proposition \ref{prop:1.1} we can assume that $M$ is a single
$A$-module. Since direct limits commute with cohomology, it suffices to prove
that 
\begin{equation} \label{eqn:258}
\begin{aligned}
& 1_{\opn{Tel}_j} \ot \opn{Hom}(u_{\bsym{a}}, 1_M) :
\opn{Tel}_j(A; \bsym{a}) \ot_A M
\\ & \quad  \qquad 
\to \opn{Tel}_j(A; \bsym{a}) \ot_A 
\opn{Hom}_A \bigl( \opn{Tel}(A; \bsym{a}), M \bigr) .
\end{aligned} 
\end{equation}
is a quasi-isomorphism for every $j$. 
Now $\opn{Tel}_j(A; \bsym{a})$ is a bounded complex of finite rank free
$A$-modules, so we can replace (\ref{eqn:258}) with 
\[ \opn{Hom}(u_{\bsym{a}}, 1_N) : N \to
\opn{Hom}_A \bigl( \opn{Tel}(A; \bsym{a}), N \bigr) , \]
where 
$N := \opn{Tel}_j(A; \bsym{a}) \ot_A M$. 
The complex $N$ satisfies the assumption of Lemma \ref{lem:tel.2}, and
therefore $\opn{Hom}(u_{\bsym{a}}, 1_N)$ is a quasi-isomorphism.
\end{proof}

\begin{lem}  \label{lem:265}
For a finite sequence $\bsym{a}$ of elements of $A$, the homomorphisms
\[ u_{\bsym{a}} \otimes 1_{\opn{Tel}} , \, 
1_{\opn{Tel}} \otimes u_{\bsym{a}} :
\opn{Tel}(A; \bsym{a}) \otimes_A \opn{Tel}(A; \bsym{a}) 
\to \opn{Tel}(A; \bsym{a}) \]
are homotopy equivalences.
\end{lem}

\begin{proof}
Because of Lemmas \ref{lem:260} and \ref{lem:201} these are quasi-isomorphisms. 
But a quasi-isomorphism between K-projective complexes is a homotopy
equivalence.
\end{proof}

\begin{prop} \label{prop:260}
Let $\a$ be a weakly proregular ideal in $A$. 
For any $M \in \cat{D}(\cat{Mod} A)$ the morphism 
\[ \tau^{\tup{L}}_{\mrm{L} \Lambda_{\a} (M)}  : \mrm{L} \Lambda_{\a} (M) \to 
\mrm{L} \Lambda_{\a} (\mrm{L} \Lambda_{\a} (M))  \]
is an isomorphism. So the functor
\[ \mrm{L} \Lambda_{\a} : 
\cat{D}(\cat{Mod} A) \to \cat{D}(\cat{Mod} A) \]
is idempotent.
\end{prop}

\begin{proof}
Choose some weakly proregular sequence $\bsym{a}$ that generates $\a$. 
According to Corollary \ref{cor:tel.1} we can replace 
$\tau^{\tup{L}}_{\mrm{L} \Lambda_{\a} (M)}$
with 
\[ \opn{Hom}(1_{T}, \opn{Hom}(u_{\bsym{a}}, 1_M)) :
\opn{Hom}_A ( T , M ) \to
\opn{Hom}_A \bigl( T , \opn{Hom}_A (T , M) \bigr) , \]
where $T := \opn{Tel}(A; \bsym{a})$.
Using Hom-tensor adjunction this can be replaced by 
\[ \opn{Hom}(1_{T} \ot u_{\bsym{a}}, 1_M ) : \opn{Hom}_A ( T , M ) \to
\opn{Hom}_A (T \ot_A T , M) . \]
By Lemma \ref{lem:265} this is a quasi-isomorphism.
\end{proof}

\begin{thm}[MGM Equivalence] \label{thm:26}
Let $A$ be a ring, and let $\a$ be a weakly proregular ideal in it. 
\begin{enumerate}
\item For any $M \in \cat{D}(\cat{Mod} A)$ one has
$\mrm{R} \Gamma_{\a} (M) \in \cat{D}(\cat{Mod} A)_{\a \tup{-tor}}$
and 
$\mrm{L} \Lambda_{\a} (M) \in \cat{D}(\cat{Mod} A)_{\a \tup{-com}}$.

\item The functor 
\[ \mrm{R} \Gamma_{\a} : \cat{D}(\cat{Mod} A)_{\a \tup{-com}} \to 
\cat{D}(\cat{Mod} A)_{\a \tup{-tor}} \]
is an equivalence, with quasi-inverse $\mrm{L} \Lambda_{\a}$.
\end{enumerate}
\end{thm}

\begin{proof}
(1) This is immediate from the idempotence of the functors 
$\mrm{R} \Gamma_{\a}$ and $\mrm{L} \Lambda_{\a}$; see Corollary
\ref{cor:11} and Proposition \ref{prop:260}.

\medskip \noindent
(2) By Lemma \ref{lem:264} and Definition \ref{dfn:8}, there are
functorial isomorphisms 
\[ M \cong \mrm{R} \Gamma_{\a} (M) \cong 
\mrm{R} \Gamma_{\a} (\mrm{L} \Lambda_{\a} (M)) \]
for $M \in \cat{D}(\cat{Mod} A)_{\a \tup{-tor}}$. 
By Lemma \ref{lem:263}  and Definition \ref{dfn:2} there are
functorial isomorphisms 
\[ N \cong \mrm{L} \Lambda_{\a} (N) \cong 
\mrm{L} \Lambda_{\a} (\mrm{R} \Gamma_{\a} (N)) \]
for $N \in \cat{D}(\cat{Mod} A)_{\a \tup{-com}}$. 
These isomorphisms set up the desired equivalence. 
\end{proof}

Here are a couple of related results. 

\begin{thm}[GM Duality] \label{thm:mgm.3}
Let $A$ be a ring, and $\a$ a weakly proregular ideal in $A$. 
For any $M, N \in \cat{D}(\cat{Mod} A)$ the morphisms 
\[ \begin{aligned}
& \opn{RHom}_A \bigl( \mrm{R} \Gamma_{\a} (M), \mrm{R} \Gamma_{\a} (N) \bigr)
\xar{\opn{RHom}(1, \sigma^{\mrm{R}}_N)}
\opn{RHom}_A \bigl( \mrm{R} \Gamma_{\a} (M), N \bigr) 
\\
& \qquad\xar{\opn{RHom}(1, \tau^{\mrm{L}}_N)}
\opn{RHom}_A \bigl( \mrm{R} \Gamma_{\a} (M), \mrm{L} \Lambda_{\a} (N) \bigr) 
\xleftarrow{\opn{RHom}(\sigma^{\mrm{R}}_M, 1)}
\\ & \qquad  \qquad
\opn{RHom}_A \bigl( M, \mrm{L} \Lambda_{\a} (N) \bigr)
\xleftarrow{\opn{RHom}(\tau^{\mrm{L}}_M, 1)}
\opn{RHom}_A \bigl( \mrm{L} \Lambda_{\a} (M), \mrm{L} \Lambda_{\a} (N) \bigr)
\end{aligned} \]
in $\cat{D}(\cat{Mod} A)$ are isomorphisms.
\end{thm}

\begin{proof}
Choose a weakly proregular sequence $\bsym{a}$ that generates $\a$, and write 
$T := \opn{Tel}(A; \bsym{a})$ and $u := u_{\bsym{a}}$.
Next choose a K-projective resolution $P \to M$ and a 
K-injective resolution $N \to I$. 
The complex $T \ot_A P$ is K-projective, and the complex 
$\opn{Hom}_A(T, I)$ is K-injective. 

By Corollary \ref{cor:tel.1} and Proposition \ref{prop:251} we can replace 
the diagram above with the diagram 
\[ \begin{aligned}
& \opn{Hom}_A \bigl( T \ot_A P, T \ot_A I \bigr)
\xar{\opn{Hom}(1, u \ot 1 )}
\opn{Hom}_A \bigl( T \ot_A P, I \bigr) 
\\
& \quad \xar{\opn{Hom}(1, \opn{Hom}( u, 1))}
\opn{Hom}_A \bigl( T \ot_A P, \opn{Hom}_A(T, I) \bigr)
\xleftarrow{\opn{Hom}( u \ot 1, 1)}
\\ & \quad  
\opn{Hom}_A \bigl( P, \opn{Hom}_A(T, I) \bigr)
\xleftarrow{\opn{Hom}(\opn{Hom}(1, u), 1)}
\opn{Hom}_A \bigl( \opn{Hom}_A(T, P), \opn{Hom}_A(T, I) \bigr)
\end{aligned} \]
in $\cat{C}(\cat{Mod} A)$. We will prove that all these morphisms are
quasi-isomorphisms. 

Consider the homomorphism of complexes 
\[ \opn{Hom}(u, 1) : 
T \ot_A P \to  \opn{Hom}_A(T, T \ot_A P) . \]
By Corollary \ref{cor:tel.1}, Proposition \ref{prop:251} and Lemma
\ref{lem:263} this is a quasi-isomorphism. 
Therefore, by Hom-tensor adjunction and the fact that $I$ is K-injective, we
see that $\opn{Hom}(1, u \ot 1)$ is a quasi-isomorphism.

By Lemma \ref{lem:265} and Hom-tensor adjunction it follows that 
$\opn{Hom}(1, \opn{Hom}( u, 1))$ and 
$\opn{Hom}(u \ot 1, 1)$ are quasi-isomorphisms.

Finally consider the homomorphism of complexes 
\[ 1 \ot \opn{Hom}(u, 1) : 
T \ot_A P \to T \ot_A \opn{Hom}_A(T, P) . \]
By Corollary \ref{cor:tel.1}, Proposition \ref{prop:251} and Lemma
\ref{lem:264} this is a quasi-isomorphism. 
Therefore, by Hom-tensor adjunction and the fact that $I$ is K-injective, we
see that $\opn{Hom}(\opn{Hom}(1, u), 1)$ is a quasi-isomorphism.
\end{proof}

\begin{cor}  \label{cor:280}
There is a functorial isomorphism
\[ \rho^{\mrm{LR}}_N : \opn{RHom}_A( \mrm{R} \Gamma_{\a} (A), N)
\iso \mrm{L} \Lambda_{\a} (N) \]
for $N \in \cat{D}(\cat{Mod} A)$, such that 
$\rho^{\mrm{LR}}_N \circ \opn{RHom}( \sigma^{\mrm{R}}_A, 1_N) = 
\tau^{\mrm{L}}_N$
as morphisms $N \to \mrm{L} \Lambda_{\a} (N)$.
\end{cor}

\begin{proof}
Take $M := A$ in Theorem \ref{thm:mgm.3}. 
\end{proof}

\begin{rem} \label{rem:historical}
Here is a brief historical survey of the material in this paper.
GM Duality for derived categories was
introduced in \cite{AJL1}.
Precursors, in ``classical'' homological algebra,
were in the papers \cite{Ma1}, \cite{Ma2} and \cite{GM}. 

The construction of the total left derived completion functor $\mrm{L}
\Lambda_{\a}$ was first done in \cite{AJL1}. Recall that \cite{AJL1} dealt with
sheaves on a  scheme $X$, where K-projective resolutions are not available, and
certain operations work only for quasi-coherent $\mcal{O}_X$-modules. 
Hence there are some technical difficulties that do not arise when working
with rings. 

The derived torsion functor goes back to work of
Grothendieck in the late 1950's (see \cite{LC} and \cite[Chapter IV]{RD}). The
use of the infinite dual Koszul complex to prove that the
functor $\mrm{R} \Gamma_{\a}$ has finite cohomological dimension already appears
in \cite{AJL1}. 

The concept of ``telescope'' comes from algebraic topology, as a device to form
the homotopy colimit in triangulated categories. This is how it was treated in 
\cite{GM}. Its purpose there was the same as in our proof of 
Theorem \ref{thm:mgm.3}. We give a concrete treatment of the telescope complex,
resulting in our Theorem \ref{thm:250}. 

GM Duality (Theorem \ref{thm:mgm.3}) was already proved in \cite{AJL1}.
Perhaps because of the complications inherent to the geometric setup, the proofs
in \cite{AJL1} are not quite transparent. Moreover, there was a subtle mistake
in \cite{AJL1} involving the concept of proregularity, that was
discovered by Schenzel (see \cite[Correction]{AJL1} and \cite{Sc}).
On the other hand, the results in the later paper \cite{Sc} are not as
strong as those in \cite{AJL1}, and this is quite confusing. See
Remark \ref{rem:280} for details. One of our aims in this paper is to clarify
the foundations of the theory in the algebraic setting.

MGM Equivalence (Theorem \ref{thm:26}) is present, in essence, already in
\cite{AJL2} and \cite{Sc}. See a discussion of the various
statements and proofs in Remark \ref{rem:280}. 

Theorem \ref{thm:26} is similar to \cite[Theorem 2.1]{DG}; but the
relationship is not clear. In \cite{DG} the authors seem to {\em define} the
derived completion and torsion functors to be 
$\opn{Hom}_A(T, M)$ and $T \ot_A M$ respectively, where 
$\bsym{a}$ is a finite sequence and 
$T := \opn{Tel}(A; \bsym{a})$. There is no apparent comparison in \cite{DG} of
these functors to the derived functors $\mrm{L} \La_{\a}(M)$ and 
$\mrm{R} \Ga_{\a}(M)$
associated to the ideal $\a$ generated by $\bsym{a}$
(something like Proposition \ref{prop:251} and Corollary \ref{cor:tel.1}).
There is also no assumption that $A$ is noetherian, nor any
mention of weak proregularity of $\bsym{a}$. The same reservations 
pertain also to \cite{DGI}. 
\end{rem}

\section{Derived Localization} \label{sec:der-loc}

In this final section we give an alternative characterization of
cohomologically complete complexes (Theorem \ref{thm:263}). This result 
is inspired by the paper \cite{KS2}. See Remark \ref{rem:derloc.1} for a
comparison. 

We make this assumption throughout the section:  
$\bsym{a} = (a_1, \ldots, a_n)$ is a weakly proregular sequence in the ring $A$,
and $\a$ is the ideal generated by $\bsym{a}$.
We do not assume that $A$ is noetherian or $\a$-adically complete.

There is an additive functor 
\[ \Gamma_{0 / \a} : \cat{Mod} A \to \cat{Mod} A \ , \ 
\Gamma_{0 / \a} (M) := M / \Gamma_{\a} (M) \ . \]
The functor $\Gamma_{0 / \a}$ has a right derived functor 
$\mrm{R} \Gamma_{0 / \a}$, constructed using K-injective resolutions. 
 
\begin{lem} \label{lem:267}
For $M \in \cat{D}(\cat{Mod} A)$ there is a distinguished triangle
\[ \mrm{R} \Gamma_{\a} (M) \xar{\sigma^{\tup{R}}_M} M \to 
\mrm{R} \Gamma_{0 / \a} (M) \distri \ , \]
in  $\cat{D}(\cat{Mod} A)$, functorial in $M$.
\end{lem}

\begin{proof}
Take any K-injective resolution $M \to I$. Consider the exact sequence 
\begin{equation*} \label{eqn:6}
 0 \to \Gamma_{\a} (I) \xar{\sigma^{}_I} I \to \Gamma_{0 / \a} (I) \to 0 
\end{equation*}
in $\cat{C}(\cat{Mod} A)$.
This gives rise to  a distinguished triangle
$\Gamma_{\a} (I) \xar{\sigma_I} I \to \Gamma_{0 / \a} (I) \distri$
in $\cat{D}(\cat{Mod} A)$, using the cone construction.
But the diagram 
$\Gamma_{\a} (I) \xar{\sigma_I} I$
is isomorphic in $\cat{D}(\cat{Mod} A)$ to the diagram 
$\mrm{R} \Gamma_{\a} (M) \xar{\sigma^{\tup{R}}_M} M$, and 
$\Gamma_{0 / \a} (I) \cong \mrm{R} \Gamma_{0 / \a} (M)$.
\end{proof}

\begin{lem} \label{lem:270}
The following conditions are equivalent
for  $M \in \cat{D}(\cat{Mod} A)$\tup{:}
\begin{enumerate}
\rmitem{i} $M$ is cohomologically $\a$-adically complete.
\rmitem{ii} $M$ is right perpendicular to $\mrm{R} \Gamma_{0 / \a} (A)$\tup{;}
namely
$\opn{RHom}_A \bigl( \mrm{R} \Gamma_{0 / \a} (A), M \bigr) = 0$.
\end{enumerate}
\end{lem}

\begin{proof}
Start with the distinguished triangle
\[ \mrm{R} \Gamma_{\a} (A) \xar{\sigma^{\tup{R}}_A} A \to 
\mrm{R} \Gamma_{0 / \a} (A) \distri \]
in $\cat{D}(\cat{Mod} A)$ that we have by Lemma \ref{lem:267}.
Now apply the functor $\opn{RHom}_A(- , M)$ to it. This gives a 
distinguished triangle
\[ \opn{RHom}_A \bigl( \mrm{R} \Gamma_{0 / \a} (A) , M \bigr) \to M 
\xar{(\sigma^{\tup{R}}_A, 1_M)}
\opn{RHom}_A \bigl( \mrm{R} \Gamma_{\a} (A) , M \bigr) \distri . \]
According to Corollary \ref{cor:280}
we can replace this triangle by the isomorphic distinguished triangle 
\begin{equation} \label{eqn:derloc.1}
\opn{RHom}_A \bigl( \mrm{R} \Gamma_{0 / \a} (A) , M \bigr) \to M 
\xar{\tau^{\tup{L}}_M} \mrm{L} \Lambda_{\a} (M)  \distri .
\end{equation}
The equivalence of the two conditions is now clear.
\end{proof}

Let $X := \opn{Spec }A$;
$Z := \opn{Spec} A / \a$, the closed subset of $X$ defined by the ideal $\a$;
$U := X - Z$, an open set in $X$; and 
$U_i :=  \opn{Spec} A[a_i^{-1}]$, the affine open set 
$\{ a_i \neq 0 \}$ of $X$. 
The collection $\bsym{U} := \{ U_i \}_{i = 1, \ldots, n}$ is
an affine open covering of $U$. 

Let $\opn{C}(\bsym{U}, \mcal{O}_X)$ be the \v{C}ech cosimplicial algebra
corresponding to this open covering. So 
\[ \opn{C}(\bsym{U}, \mcal{O}_X)^p = 
\prod_{1 \leq i_0 \leq \cdots \cdots \leq i_p \leq n} \, 
\Ga( U_{i_0} \cap \cdots \cap U_{i_p}, \mcal{O}_X) . \]
Note that 
\[ \Ga( U_{i_0} \cap \cdots \cap U_{i_p}, \mcal{O}_X) \cong 
A[(a_{i_0} \cdots a_{i_p})^{-1}] \]
as $A$-algebras. 

Any cosimplicial algebra $C$ has the standard normalization $\opn{N}(C)$, which
is a DG algebra. In degree $p$ the abelian group 
$\opn{N}(C)^p$ is the kernel of all the codegeneracy operators. 
The multiplication is by the Alexander-Whitney formula (which is usually
noncommutative!), and the differential is the alternating sum of the coboundary
operators. For full details see any book on simplicial methods; or 
\cite[Section 1]{HY}.

\begin{dfn}
Let $\opn{C}(A; \bsym{a}) := \opn{N}(\opn{C}(\bsym{U}, \mcal{O}_X))$, the
standard normalization of the cosimplicial algebra 
$\opn{C}(\bsym{U}, \mcal{O}_X)$. The DG $A$-algebra $\opn{C}(A; \bsym{a})$ is
called the {\em derived localization} of $A$ with respect to the sequence of
elements $\bsym{a}$.
\end{dfn}

For the convenience of the reader we shall now give explicit formulas for the
structure of the DG algebra $\opn{C}(A; \bsym{a})$. 
For any $q \in \{ 0, \ldots, n - 1 \}$ consider the set of strictly increasing
sequences
$\bsym{k} = (k_0, \ldots, k_q)$
in $\{ 1, \ldots, n \}^{q + 1}$.
Define
\[ \opn{C}^{\bsym{k}}(A; \bsym{a}) := 
A[ (a_{k_0} \cdots a_{k_q})^{-1} ] . \]
This is a commutative $A$-algebra, isomorphic to 
\[ A[a_{k_0}^{-1}]  \otimes_A \cdots \otimes_A A[a_{k_q}^{-1}] . \]
If $\bsym{k}$ is a subsequence of $\bsym{l}$ then there is a canonical
$A$-algebra homomorphism
\[ \phi^{\bsym{k}, \bsym{l}} :
\opn{C}^{\bsym{k}}(A; \bsym{a}) \to \opn{C}^{\bsym{l}}(A; \bsym{a}) . \]
We define 
\begin{equation} \label{eqn:7}
\opn{C}^q(A; \bsym{a}) := \prod_{\bsym{k}} \, \opn{C}^{\bsym{k}}(A; \bsym{a}) 
\end{equation}
where $\bsym{k} = (k_0, \ldots, k_q)$ is strictly increasing.
The differential
\[ \d : \opn{C}^q(A; \bsym{a}) \to \opn{C}^{q+1}(A; \bsym{a}) \]
has components
\[ \d^{\bsym{k}, \bsym{l}} : \opn{C}^{\bsym{k}}(A; \bsym{a})
\to \opn{C}^{\bsym{l}}(A; \bsym{a}) \]
for $\bsym{l} = (l_0, \ldots, l_{q+1})$, with
\[ \d^{\bsym{k}, \bsym{l}} :=
\begin{cases}
(-1)^j \phi^{\bsym{k}, \bsym{l}} & \text{ if } \bsym{k} \text{ is gotten from }
\bsym{l} \text{ by deleting } l_j \\
0 & \text{ otherwise } . 
\end{cases} \]

Here is the explicit description of the Alexander-Whitney multiplication on
$\opn{C}(A; \bsym{a})$. For strictly increasing multi-indices
$\bsym{k} = (k_0, \ldots, k_p)$ and 
$\bsym{l} = (l_0, \ldots, l_{q})$, the multiplication
\[ * : \opn{C}^{\bsym{k}}(A; \bsym{a}) \times 
\opn{C}^{\bsym{l}}(A; \bsym{a}) \to \opn{C}^{p + q}(A; \bsym{a}) \]
is this: if $k_p = l_0$ then let
\[ \bsym{k} \smallsmile \bsym{l} := (k_0, \ldots, k_p, 
l_1, \ldots, l_{q}) . \]
There are $A$-algebra homomorphisms
\[ \phi^{\bsym{k}, \bsym{k} \smallsmile \bsym{l}} : 
\opn{C}^{\bsym{k}}(A; \bsym{a}) \to 
\opn{C}^{\bsym{k} \smallsmile \bsym{l}}(A; \bsym{a}) \]
and
\[ \phi^{\bsym{l}, \bsym{k} \smallsmile \bsym{l}} : 
\opn{C}^{\bsym{l}}(A; \bsym{a}) \to 
\opn{C}^{\bsym{k} \smallsmile \bsym{l}}(A; \bsym{a}) . \]
For elements
$a \in \opn{C}^{\bsym{k}}(A; \bsym{a})$ and 
$b \in \opn{C}^{\bsym{l}}(A; \bsym{a})$
we let
\[ a * b := \phi^{\bsym{k}, \bsym{k} \smallsmile \bsym{l}}(a) \cdot 
\phi^{\bsym{l}, \bsym{k} \smallsmile \bsym{l}}(b) 
\in \opn{C}^{\bsym{k} \smallsmile \bsym{l}}(A; \bsym{a}) . \]
If $k_p \neq l_0$ then the multiplication $*$ is zero.

Let us denote by $f_{\bsym{a}} : A \to \opn{C}^0(A; \bsym{a})$
the canonical ring homomorphism. It is easy to check that this becomes a
homomorphism of DG algebras 
\begin{equation}
f_{\bsym{a}} : A \to \opn{C}(A; \bsym{a}) .
\end{equation}

Observe that $\opn{C}(A; \bsym{a})$ is concentrated in
degrees $0, \ldots, n-1$; and each 
$\opn{C}(A; \bsym{a})^p$
is a flat $A$-module. 
If $n = 1$ then $\opn{C}(A; \bsym{a}) = A[a_1^{-1}]$; in particular
in this case $\opn{C}(A; \bsym{a})$ is commutative. However if $n \geq 2$ 
then the center of $\opn{C}(A; \bsym{a})$ is the ring 
$\Ga(U, \mcal{O}_X)$, which is many cases (e.g.\ $A := K[t_1, t_2]$, a
polynomial algebra over a field, and $\bsym{a} := (t_1, t_2)$)
equals $A$. 

\begin{lem} \label{lem:derloc.1} 
\begin{enumerate}
\item There is an isomorphism 
$\opn{K}^{\vee}_{\infty}(A; \bsym{a})[1] \cong \opn{cone}(f_{\bsym{a}})$
in \lb $\cat{C}(\cat{Mod} A)$. The corresponding distinguished
triangle in $\cat{K}(\cat{Mod} A)$ is 
\[ \opn{K}^{\vee}_{\infty}(A; \bsym{a}) 
\xar{e^{\vee}_{\bsym{a}, \infty}} A
\xar{f_{\bsym{a}}} \opn{C}(A; \bsym{a}) \distri . \] 

\item The homomorphisms 
\[ 1_{\opn{C}} \otimes f_{\bsym{a}}, \, f_{\bsym{a}} \otimes 1_{\opn{C}} :
\opn{C}(A; \bsym{a}) \to \opn{C}(A; \bsym{a}) \otimes_A \opn{C}(A; \bsym{a})
\]
are quasi-isomorphisms.
\end{enumerate}
\end{lem}

\begin{proof}
(1) This is a direct calculation, quite easy. 

\medskip \noindent
(2) Since the complexes in the distinguished triangle in part (1) are all K-flat
over $A$, the assertion follows from Lemma \ref{lem:260}.
\end{proof}

\begin{thm} \label{thm:263}
Let $\bsym{a} = (a_1, \ldots, a_n)$ be a weakly proregular sequence in the ring
$A$, and let $\a$ be the ideal generated by $\bsym{a}$.
The following conditions are equivalent
for $M \in \cat{D}(\cat{Mod} A)$\tup{:}
\begin{enumerate}
\rmitem{i} $M$ is  cohomologically $\a$-adically complete.
\rmitem{ii} $\opn{RHom}_A \bigl( \opn{C}(A; \bsym{a}) , M \bigr) = 0$.
\end{enumerate}
\end{thm}

\begin{proof}
{}From Lemma \ref{lem:derloc.1}(1), Lemma
\ref{lem:267} and Corollary \ref{cor:kosz.1} (applied to $M := A$)
we see that there is an isomorphism
$\mrm{R} \Gamma_{0 / \a} (A) \cong \opn{C}(A; \bsym{a})$
in $\cat{D}(\cat{Mod} A)$. 
Now combine this with Lemma  \ref{lem:270}.
\end{proof}

Let $F : \cat{D} \to \cat{D}'$ be an additive functor between additive
categories. Recall that the {\em essential
image} of $F$ is the full subcategory of $\cat{D}'$ on the objects 
$N' \in \cat{D}'$ such that $N' \cong F (N)$ for some $N \in \cat{D}$. 
The {\em kernel} of $F$ is the full subcategory of $\cat{D}$ on the objects 
$N \in \cat{D}$ such that $F (N) \cong 0$. 

\begin{prop} \label{prop:derloc.1}
When $\a$ is a weakly proregular ideal in $A$, the kernel of the
functor $\mrm{L} \Lambda_{\a}$ equals the kernel of the
functor $\mrm{R} \Gamma_{\a}$. 
\end{prop}

\begin{proof}
This is an immediate consequence of the MGM Equivalence (Theorem \ref{thm:26}).
\end{proof}

For a DG algebra $C$ we denote by $\cat{DGMod} C$ the category of left DG
$C$-modules, and by $\til{\cat{D}}(\cat{DGMod} C)$ the derived category, gotten
by inverting the quasi-iso\-morph\-isms in $\cat{DGMod} C$. 

\begin{thm} \label{thm:derloc.1}
Let $\bsym{a} = (a_1, \ldots, a_n)$ be a weakly proregular sequence in the ring
$A$, and let $\a$ be the ideal generated by $\bsym{a}$. Consider the
triangulated functor 
\[ F : \til{\cat{D}} (\cat{DGMod} \opn{C}(A; \bsym{a})) \to 
\cat{D}(\cat{Mod} A) \]
induced by the DG algebra homomorphism 
$A \to \opn{C}(A; \bsym{a})$. 
\begin{enumerate}
\item The functor $F$ is full and faithful.
\item The essential image of $F$ equals the kernel of the functor 
$\mrm{L} \Lambda_{\a}$. 
\end{enumerate}
\end{thm}

\begin{proof}
(1) Let's write $C := \opn{C}(A; \bsym{a})$,
$\cat{D}(C) := \til{\cat{D}} (\cat{DGMod} C)$
and
$\cat{D}(A) := \cat{D}(\cat{Mod} A)$.
Take any $N \in \cat{DGMod} C$. Lemma \ref{lem:derloc.1}(2) implies that 
$f_{\bsym{a}} \otimes 1_N : N \to C \otimes_A N$
is a quasi-isomorphism. This shows that the functor 
$G : \cat{D}(A) \to \cat{D}(C)$, 
$G (M) := C \otimes_A M$, is right adjoint to $F$, and it satisfies
$G \circ F \cong \bsym{1}_{\cat{D}(C)}$. Hence $F$ is fully faithful. 

\medskip \noindent
(2) Let's write $K := \opn{K}^{\vee}_{\infty}(A; \bsym{a})$.
Take any $M \in \cat{D}(A)$.
In view of the idempotence of $C$ (namely Lemma \ref{lem:derloc.1}(2)),
Proposition \ref{prop:derloc.1}, Corollary \ref{cor:kosz.1} and the
proof of part (1) above, it is enough to show that 
$K \otimes_A M \cong 0$
iff $M \cong C \otimes_A M$. 
Now after applying $- \otimes_A M$ to the distinguished triangle in
Lemma \ref{lem:derloc.1}(1) we obtain a distinguished triangle
\[ K \otimes_A M \to M \to C \otimes_A M  \distri \] 
in $\cat{D}(A)$. So the conditions are indeed equivalent.
\end{proof}

\begin{rem}
One can show that 
$\cat{D}(A)_{\a \tup{-tor}}$ is a Bousfield localization of $\cat{D}(A)$
in the sense of \cite[Chapter 9]{Ne}. Here
we use the notation from the proof above. Therefore, using Proposition
\ref{prop:derloc.1} and Theorem \ref{thm:derloc.1}, we see that there is an
exact sequence of triangulated categories
\[ 0 \to \cat{D}(C) \xar{F} \cat{D}(A) \xar{\mrm{R} \Gamma_{\a}} 
\cat{D}(A)_{\a \tup{-tor}} \to  0 . \]
This was already observed in \cite[Remark 0.4]{AJL1} and \cite{DG}.

Moreover, Jorgensen, in \cite[Theorem 3.3]{Jo}, showed that when we take into
account the functors 
\[ \opn{inc}, \mrm{L} \Lambda_{\a} : \cat{D}(A)_{\a \tup{-tor}} \to 
\cat{D}(A) \]
and
\[ \opn{RHom}_A(C, -) , \,  C \ot_{A} - : \cat{D}(A) \to \cat{D}(C) \]
we obtain a recollement of triangulated categories.
\end{rem}

\begin{rem}
The scheme $U = X - Z$ is quasi-affine. We denote by
$\cat{QCoh} \mcal{O}_U$ the category of quasi-coherent $\mcal{O}_U$-modules. 
It can be shown that there is a canonical $A$-linear equivalence of
triangulated categories
\[ \cat{D}(\cat{QCoh} \mcal{O}_U) \approx
\til{\cat{D}}(\cat{DGMod} \opn{C}(A; \bsym{a})) . \]
Of course in the principal case ($n = 1$) this is a trivial fact. 
In terms of derived Morita theory, the equivalence
above corresponds to the fact that $\mcal{O}_U$ is a compact generator of 
$\cat{D}(\cat{QCoh} \mcal{O}_U)$. 
\end{rem}

\begin{rem} \label{rem:derloc.1}
In the paper \cite{KS2} the authors consider the special case where 
$\a$ is a principal ideal of $A$, generated by a regular element (i.e.\ a 
non-zero-divisor) $a$. Here the derived localization $\opn{C}(A; a)$ is just
the  commutative ring $A[a^{-1}]$, and the notation of \cite{KS2} for this
algebra is $A^{\mrm{loc}}$. Theorems \ref{thm:263}  and  \ref{thm:derloc.1} for
this case are closely related to \cite[Corollaries 1.5.7 and 1.5.9]{KS2}. 
The {\em Cohomological Nakayama Theorem} in \cite{PSY2} is inspired by results
in \cite{KS2}.
\end{rem}


\end{document}